\documentclass[11pt,reqno]{amsart}
\usepackage{amsmath,amssymb,amsthm,mathrsfs,enumerate,bm,xcolor,multirow,pbox}
\usepackage{graphicx,color,framed,tikz}
\usepackage[colorlinks,linkcolor=red,citecolor=blue,urlcolor=blue]{hyperref}
\usepackage{enumitem}
\setlist{leftmargin=5mm}
\allowdisplaybreaks[4]
\numberwithin{equation}{section}
\newcommand{\N}{\mathbb{N}}
\newcommand{\R}{\mathbb{R}}

\newcommand{\E}{\mathbb{E}}
\newcommand{\Prob}{\mathbb{P}}
\newcommand{\G}{\mathbb{G}}

\newcommand{\pnorm}[2]{\lVert#1\rVert_{#2}}

\newcommand{\biggpnorm}[2]{\bigg\lVert#1\bigg\rVert_{#2}}
\newcommand{\abs}[1]{\lvert#1\rvert}
\newcommand{\bigabs}[1]{\big\lvert#1\big\rvert}
\newcommand{\biggabs}[1]{\bigg\lvert#1\bigg\rvert}

\renewcommand{\epsilon}{\varepsilon}

\renewcommand{\d}[1]{\mathrm{d}#1}

\newcommand{\floor}[1]{\left\lfloor #1 \right\rfloor}

\renewcommand{\tilde}{\widetilde}
\let\liminf\relax
\DeclareMathOperator*\liminf{\underline{lim}}
\let\limsup\relax
\DeclareMathOperator*\limsup{\overline{lim}}
\DeclareMathOperator*{\argmax}{arg\,max}
\DeclareMathOperator*{\argmin}{arg\,min}

\AtBeginDocument{%
	\def\MR#1{}
}

\theoremstyle{definition}\newtheorem{problem}{Problem}[section]
\theoremstyle{definition}
\theoremstyle{remark}\newtheorem{assumption}{Assumption}

\theoremstyle{remark}\newtheorem{remark}[problem]{Remark}
\theoremstyle{definition}
\theoremstyle{plain}\newtheorem{theorem}[problem]{Theorem}
\theoremstyle{plain}
\theoremstyle{plain}\newtheorem{lemma}[problem]{Lemma}
\theoremstyle{plain}\newtheorem{proposition}[problem]{Proposition}
\theoremstyle{plain}\newtheorem{corollary}[problem]{Corollary}
\theoremstyle{plain}

\begin{document}

\title[Set structured ERMs are rate optimal]{Set structured global empirical risk minimizers are rate optimal in general dimensions}
\thanks{The research of Q. Han is partially supported by NSF DMS-1916221.}

\author[Q. Han]{Qiyang Han}

\address[Q. Han]{
Department of Statistics, Rutgers University, Piscataway, NJ 08854, USA.
}
\email{qh85@stat.rutgers.edu}

\date{\today}

\maketitle

\begin{abstract}
Entropy integrals are widely used as a powerful empirical process tool to obtain upper bounds for the rates of convergence of global empirical risk minimizers (ERMs), in standard settings such as density estimation and regression. The upper bound for the convergence rates thus obtained typically matches the minimax lower bound when the entropy integral converges, but admits a strict gap compared to the lower bound when it diverges. Birg\'e and Massart \cite{birge1993rates} provided a striking example showing that such a gap is real with the entropy structure alone: for a variant of the natural H\"older class with low regularity, the global ERM actually converges at the rate predicted by the entropy integral that substantially deviates from the lower bound. The counter-example has spawned a long-standing negative position on the use of global ERMs in the regime where the entropy integral diverges, as they are heuristically believed to converge at a sub-optimal rate in a variety of models.

The present paper demonstrates that this gap can be closed if the models admit certain degree of `set structures' in addition to the entropy structure. In other words, the global ERMs in such set structured models will indeed be rate-optimal, matching the lower bound even when the entropy integral diverges. The models with set structures we investigate include (i) image and edge estimation, (ii) binary classification, (iii) multiple isotonic regression, (iv) $s$-concave density estimation, all in general dimensions when the entropy integral diverges. Here set structures are interpreted broadly in the sense that the complexity of the underlying models can be essentially captured by the size of the empirical process over certain class of measurable sets, for which matching upper and lower bounds are obtained to facilitate the derivation of sharp convergence rates for the associated global ERMs.
\end{abstract}


\section{Introduction}

\subsection{Overview}

Empirical risk minimization (ERM) is one of the most widely used statistical procedures for the purpose of estimation and inference. Theoretical properties for various ERMs, in particular in terms of rates of convergence, have been intensively investigated by various authors \cite{birge1983approximation,birge1993rates,barron1999risk,van1987new,van1990estimating,van1993hellinger,van1995method,wong1995probability,van1996weak,van2000empirical,koltchinskii2006local}, in a number of by-now standard settings. To motivate our discussion, let us focus on the standard Gaussian regression setting: Let $X_1,\ldots,X_n$ be i.i.d. covariates taking value in $(\mathcal{X},\mathcal{A})$ with law $P$, and the responses $Y_i$'s are given by
\begin{align}\label{model:reg}
Y_i = f_0(X_i)+\xi_i,\quad i=1,\ldots,n,
\end{align}
where $\xi_i$'s are i.i.d. $\mathcal{N}(0,1)$, and $f_0 $ belongs to a uniformly bounded class  $\mathcal{F} \subset L_\infty(1)$. One canonical global ERM in the regression model (\ref{model:reg}) is the least squares estimator (LSE):
\begin{align*}
\widehat{f}_n \in \argmin_{f \in \mathcal{F}} \sum_{i=1}^n (Y_i-f(X_i))^2.
\end{align*}
The performance of $\widehat{f}_n$ is usually evaluated through the risk under squared $L_2$ loss  $\E_{f_0} \pnorm{\widehat{f}_n-f_0}{L_2(P)}^2$, or its `probability' version.

The seminal work of Birg\'e and Massart \cite{birge1993rates} (and other references cited above) shows that an upper bound $\bar{r}_n^2$  for the risk $\E_{f_0} \pnorm{\widehat{f}_n-f_0}{L_2(P)}^2$ can be obtained by solving
\begin{align}\label{ineq:upper_bound_risk}
\int_{c\bar{r}_n^2}^{\bar{r}_n} \sqrt{\log \mathcal{N}_{[\,]} (\epsilon, \mathcal{F}, L_2(P)) }\ \d{\epsilon} \asymp \sqrt{n}\cdot \bar{r}_n^2.
\end{align}
Here $\mathcal{N}_{[\,]}(\epsilon,\mathcal{F},L_2(P))$ is the $\epsilon$-bracketing number of $\mathcal{F}$ under $L_2(P)$. On the other hand, a lower bound $\underline{r}_n^2$ for the risk, often evaluated in a minimax framework, i.e. $\inf_{\widetilde{f}_n}\sup_{f_0 \in \mathcal{F}} \E_{f_0} \pnorm{\widetilde{f}_n-f_0}{L_2(P)}^2\geq \underline{r}_n^2$, can be obtained (cf. \cite{birge1983approximation,yang1999information}) via a different equation
\begin{align}\label{ineq:lower_bound_risk}
\underline{r}_n \sqrt{ \log \mathcal{N}(\underline{r}_n, \mathcal{F}, L_2(P))}\asymp \sqrt{n}\cdot \underline{r}_n^2,
\end{align}
where $\mathcal{N}(\epsilon,\mathcal{F},L_2(P))$ is the $\epsilon$-covering number of $\mathcal{F}$ under $L_2(P)$. Note that the left hand side of (\ref{ineq:upper_bound_risk}) is no smaller than the left hand side of (\ref{ineq:lower_bound_risk}), so we always have $\underline{r}_n\lesssim \bar{r}_n$. Suppose for now that the difference in the covering and bracketing entropy can be ignored, and it holds for some $\alpha>0$ that
\begin{align}\label{ineq:L2_entropy_intro}
\log \mathcal{N}(\epsilon,\mathcal{F},L_2(P))\asymp \log \mathcal{N}_{[\,]}(\epsilon,\mathcal{F},L_2(P)) \asymp \epsilon^{-2\alpha}.
\end{align}
The parameter $\alpha>0$ measures the complexity of $\mathcal{F}$, and is closely related to the `smoothness' of $\mathcal{F}$, cf. \cite{van1996efficient,van2000empirical,gine2015mathematical}. Solving the equations (\ref{ineq:upper_bound_risk}) and (\ref{ineq:lower_bound_risk}) yields that
\begin{align}\label{def:rate_intro}
\underline{r}_n\asymp n^{-\frac{1}{2(1+\alpha)}},\quad \bar{r}_n \asymp  \big(n^{-\frac{1}{2(1+\alpha)}} \vee n^{-\frac{1}{4\alpha}}\big)\sqrt{\log^{\bm{1}(\alpha=1)} n}.
\end{align}
Modulo the logarithmic factor in the boundary case $\alpha=1$, we see a somewhat strange phase-transition phenomenon:
\begin{itemize}
	\item If $\alpha \in (0,1)$, $
	0<\liminf_n \bar{r}_n/\underline{r}_n \leq \limsup_n \bar{r}_n/\underline{r}_n<\infty$. In this regime, $\mathcal{F}$ is \emph{Donsker} since a central limit theorem in $\ell^\infty(\mathcal{F})$ holds for the empirical process due to the convergence of the bracketing entropy integral at $0$, cf. \cite[Section 2.5.2]{van1996weak}.
	\item If $\alpha>1$, $
	\liminf_n \bar{r}_n/\underline{r}_n = \infty$. In this regime, $\mathcal{F}$ is \emph{non-Donsker} since there does not exist a central limit theorem in $\ell^\infty(\mathcal{F})$ for the empirical process---the limiting Brownian bridge process indexed by $\mathcal{F}$ is not sample bounded.
\end{itemize}
Although at this point (\ref{ineq:upper_bound_risk}) only gives an upper bound for $\bar{r}_n$, Birg\'e and Massart \cite{birge1993rates} showed by a stunning example that in the regime $\alpha>1$, $\bar{r}_n$ can actually be attained (up to logarithmic factors) for the global ERM (called `minimum contrast estimators' therein) over a slightly constructed $\mathcal{F}$ based on H\"older classes on $[0,1]$ with smoothness less than $1/2$. Consequently, there is a genuine gap between the upper bound $\bar{r}_n$ and the lower bound $\underline{r}_n$ obtained from general empirical process techniques based on $L_2$ entropy structures (\ref{ineq:upper_bound_risk}) and (\ref{ineq:lower_bound_risk}) alone, in the regime $\alpha>1$. 

The counter-example in \cite{birge1993rates} results in a long-standing negative position on the use of global ERMs in the regime $\alpha>1$, as they are heuristically believed to be rate-suboptimal in various problems falling into the non-Donsker regime, beyond the natural setting of H\"older-type smoothness classes, cf. \cite{van2000empirical,seregin2010nonparametric,guntuboyina2012optimal,kim2016global}, just to name a few references. A common (but perhaps vague) heuristic is that when $\alpha>1$, the class $\mathcal{F}$ is too `massive' for global ERMs to achieve the optimal rate.

It should be mentioned that the rate sub-optimality phenomenon is due to the \emph{global} nature of ERM that searches over the entire parameter space, since it is easy to construct a `theoretical' rate-optimal estimator by searching over certain maximal packing sets of $\mathcal{F}$ even in the regime $\alpha>1$ (usually known as the `sieve' estimator \cite{grenander1981abstract,lecam1973convergence}). For instance, the LSE over a maximal $\underline{r_n}$-packing set of $\mathcal{F}$ typically leads to the desired optimal rate of convergence. Such a theoretical construction often occurs in a minimax approach for a given statistical model, cf. \cite{mammen1995asymptotical,guntuboyina2012optimal,brunel2013adaptive}.

At a deeper level from the perspective of empirical process theory, the upper bound (\ref{ineq:upper_bound_risk}) comes from the Dudley's entropy integral, and the lower bound (\ref{ineq:lower_bound_risk}) is inherited with Sudakov minorization. From the recent work \cite{saumard2012optimal,chatterjee2014new,van2015concentration,han2017sharp}, it is now understood that the  risk $r_n^2\equiv \E_{f_0}\pnorm{\widehat{f}_n-f_0}{L_2(P)}^2$ can be \emph{completely} characterized (at least in the simple Gaussian regression model with uniformly bounded $\mathcal{F}$), by the following (not fully rigorous but essential)\footnote{Rigorously, $r_n$ is determined by the location of the maxima of the map $r\mapsto \sup_{f \in \mathcal{F}-f_0: \pnorm{f}{L_2(P)}\leq r} \G_n(\xi f-f^2)-\sqrt{n}r^2$ provided it exists uniquely, cf. \cite{van2015concentration,han2017sharp}. For Gaussian errors $\xi_i$'s and uniformly bounded $\mathcal{F}$, the order of $r_n$ can typically be obtained by matching upper and lower moment estimates for the LHS of (\ref{ineq:risk_characterization}).} equation:
\begin{align}\label{ineq:risk_characterization}
\E \sup_{f \in \mathcal{F}-f_0: \pnorm{f}{L_2(P)}\leq r_n} \abs{\G_n(f)} \asymp \sqrt{n}\cdot r_n^2.
\end{align}
Here $\G_n \equiv \sqrt{n}(\Prob_n-P)$ is the empirical process. Since Dudley's entropy integral provides an upper bound, while the Sudakov minorization gives a lower bound, for the empirical process in (\ref{ineq:risk_characterization}) as soon as it enters the `Gaussian domain' (= for $n$ large in our case), the only possibility for which $\bar{r}_n$ and $\underline{r}_n$ do not match lies in situations where the entropy integral bound deviates substantially from the Sudakov minorization. This is indeed the case in the non-Donsker regime $\alpha>1$: under a variant of the $L_2$ entropy condition (\ref{ineq:L2_entropy_intro}) (see (\ref{ineq:L2_entropy})), standard bounds lead to the estimates
\begin{align}\label{ineq:ep_size_classic_nondonsker}
n^{(\alpha-1)/2(\alpha+1)}\lesssim \E \sup_{f \in \mathcal{F}} \abs{\G_n(f)}&\lesssim n^{(\alpha-1)/2\alpha}.
\end{align}
The upper and lower bounds above do not match, and neither of them can be improved without further structural assumptions. In particular, (\ref{ineq:ep_size_classic_nondonsker}) leads to the discrepancy between $\bar{r}_n$ and $\underline{r}_n$ in (\ref{def:rate_intro}) for non-Donsker $\mathcal{F}$'s.

Despite the strong suspicion in the literature (cited above) that the actual rate $r_n$ of global ERMs will likely match $\bar{r}_n$ which has a strict gap compared to the minimax lower bound $\underline{r}_n$, there appears recently some surprising special cases in which global ERMs are proved to be rate optimal even in the regime $\alpha>1$. One example is given by the multiple isotonic regression model studied by the author in \cite{han2017isotonic}. When $d\geq 3$, by the entropy estimate in \cite{gao2007entropy}, the class of multiple isotonic functions is in the non-Donsker regime $\alpha>1$, but interestingly \cite{han2017isotonic} proved that the natural LSE (= global ERM) is still minimax rate-optimal (up to logarithmic factors) in $L_2$ loss. The proof techniques in \cite{han2017isotonic} are rather intricate and somewhat indirect, so they unfortunately do not shed light on why the LSE must be rate-optimal (see Remark \ref{rmk:isotonic_reg} for more technical details).

The purpose of the present paper is to demonstrate a general underlying mechanism for the rate-optimality phenomenon for global ERMs beyond the isotonic LSE in general dimensions as mentioned above. This amounts to the identification of a sub-family of $\mathcal{F}$'s satisfying the $L_2$ entropy condition (\ref{ineq:L2_entropy_intro}) (or see its variant (\ref{ineq:L2_entropy})), in which the the associated global ERMs remain rate-optimal. As one may expect, the key step is to close the gap between the upper and lower bounds in (\ref{ineq:ep_size_classic_nondonsker}) under suitable structural assumptions on $\mathcal{F}$, and for the purpose of rate-optimality, one should aim at improving the upper bound $n^{(\alpha-1)/2\alpha}$ to match the lower bound $n^{(\alpha-1)/2(\alpha+1)}$ in (\ref{ineq:ep_size_classic_nondonsker}). We show (cf. Theorem \ref{thm:upper_bound_ep}) that this is indeed possible for $\mathcal{F}\subset L_\infty(1)$ in the non-Donsker regime $\alpha>1$, under a stronger $L_1$ entropy condition:
\begin{align}\label{ineq:L1_entropy_intro}
\log \mathcal{N}_{[\,]}(\epsilon^2, \mathcal{F},L_1(P)) \lesssim \epsilon^{-2\alpha}.
\end{align}
An important class of $\mathcal{F}$ that verifies (\ref{ineq:L1_entropy_intro}) is the class of indicators indexed by a class of measurable sets.  More specifically, for a class of measurable sets $\mathscr{C}$, as the $L_1(P)$ size of any element $\bm{1}_C$ where $C \in \mathscr{C}$ is the same as its squared $L_2(P)$ size, (\ref{ineq:L1_entropy_intro}) is automatically verified provided the $L_2$ entropy condition (\ref{ineq:L2_entropy_intro}) (or (\ref{ineq:L2_entropy})) is satisfied. Therefore, under the prescribed $L_2$ entropy condition alone, as long as the $L_2$-size of $\mathscr{C}$ is not too small, it holds for $\alpha\neq 1$ that
\begin{align}\label{ineq:size_ep_intro}
\E \sup_{C \in \mathscr{C}(\sigma)}\abs{\G_n(C)} \asymp \max\{ \sigma^{1-\alpha}, n^{(\alpha-1)/2(\alpha+1)} \}.
\end{align}
Here $\mathscr{C}(\sigma)\equiv\{C \in \mathscr{C}: P(C)\leq \sigma^2\}$, and for a measurable set $C$, $\G_n(C)\equiv \G_n(\bm{1}_C)$. For $\alpha>1$, the empirical process (\ref{ineq:size_ep_intro}) is in the non-Donsker regime, and our estimate (\ref{ineq:size_ep_intro}) is still sharp in this challenging regime thanks to the improved estimate of the generic bound (\ref{ineq:ep_size_classic_nondonsker}) due to the $L_1$ entropy condition (\ref{ineq:L1_entropy_intro}).

In light of (\ref{ineq:size_ep_intro}), we will show that in models with certain ``set structures'', the global ERMs will achieve the minimax optimal rates of convergence, i.e. $\bar{r}_n\asymp r_n \asymp \underline{r}_n$ (up to logarithmic factors) even in the regime $\alpha>1$ when the entropy integral diverges. Here ``set structures'' are interpreted broadly in the sense that the size of the underlying empirical process (\ref{ineq:risk_characterization}) indexed by $\mathcal{F}$ can be characterized by an empirical process indexed by certain class of measurable sets for which an estimate of the type (\ref{ineq:size_ep_intro}) is possible. This concept will be illustrated throughout a detailed study on the risk behavior (or rates of convergence) for the natural global ERMs in the following models:
\begin{enumerate}
	\item Image and edge estimation;
	\item Binary classification;
	\item Multiple isotonic regression (revisited);
	\item $s$-concave density estimation,
\end{enumerate}
all of which will be considered in general dimensions, where the problems necessarily fall into the non-Donsker regime $\alpha>1$. In the special case of the multiple isotonic regression model, our new techniques present a much easier and intuitive proof (compared to the previous work \cite{han2017isotonic}) that explains the reason why the natural least squares estimator is indeed rate minimax (up to logarithmic factors) for $d\geq 3$---the complexity of the isotonic LSE is captured by that of the class of upper and lower sets that arise naturally in the min-max representation of the isotonic LSE, cf. \cite{robertson1988order}.

\subsection{Related works}

Prior upper bounds for the empirical process (\ref{ineq:size_ep_intro}) in the regime $\alpha>1$ are obtained in, e.g. \cite[Theorem 2]{dudley1982empirical} or \cite[Theorem 11.4]{dudley1999uniform}, with additional logarithmic factors and in the weaker `in probability' form. This paper provides stronger matching upper and lower bounds for the expected supremum of the empirical process in (\ref{ineq:size_ep_intro}), and in fact the bounds hold in much greater generality; see Theorem \ref{thm:upper_bound_ep} for precise statements. 

The major part of this work is based on Chapter 4 of the author's University of Washington Ph.D. thesis in 2018. During the preparation the paper, the author becomes aware of the very nice work \cite{dagan2019log} which derives, among other things, global risk bounds for the log-concave (=$0$-concave) maximum likelihood estimators (MLEs) based on a reduction scheme of \cite{carpenter2018near} and an upper bound similar to (\ref{ineq:size_ep_intro}).  Here we prove that the rate-optimality in example (4) holds for the maximum regime of $s$ in which the $s$-concave MLE exists. See Remark \ref{rmk:s_concave_reduction} for more technical remarks.

\subsection{Organization}
The rest of the paper is organized as follows. Section \ref{section:ep_sharp_bounds} is devoted to the new upper and lower bounds for the size of the empirical process indexed by a class of functions satisfying certain special entropy conditions that includes the class of measurable sets. Applications of these new bounds to the models mentioned above are detailed in Section \ref{section:application}. For clarity of presentation, proofs are deferred to Sections \ref{section:proof_ep}-\ref{section:proof_application}, and the Appendix.

\subsection{Notation}\label{section:notation}
For a real-valued random variable $\xi$ and $1\leq p<\infty$, let $\pnorm{\xi}{p} := \big(\E\abs{\xi}^p\big)^{1/p} $ denote the ordinary $p$-norm. 

For a real-valued measurable function $f$ defined on $(\mathcal{X},\mathcal{A},P)$, $\pnorm{f}{L_p(P)}\equiv \pnorm{f}{P,p}\equiv \big(P\abs{f}^p)^{1/p}$ denotes the usual $L_p$-norm under $P$, and $\pnorm{f}{\infty}\equiv \sup_{x \in \mathcal{X}} \abs{f(x)}$. $f$ is said to be $P$-centered if $Pf=0$. $L_p(g,B)$ denotes the $L_p(P)$-ball centered at $g$ with radius $B$. For simplicity we write $L_p(B)\equiv L_p(0,B)$. 

Throughout the article $\epsilon_1,\ldots,\epsilon_n$ will be i.i.d. Rademacher random variables independent of all other random variables. $C_{x}$ will denote a generic constant that depends only on $x$, whose numeric value may change from line to line unless otherwise specified. $a\lesssim_{x} b$ and $a\gtrsim_x b$ mean $a\leq C_x b$ and $a\geq C_x b$ respectively, and $a\asymp_x b$ means $a\lesssim_{x} b$ and $a\gtrsim_x b$ [$a\lesssim b$ means $a\leq Cb$ for some absolute constant $C$]. For two real numbers $a,b$, $a\vee b\equiv \max\{a,b\}$ and $a\wedge b\equiv\min\{a,b\}$. Slightly abusing notation, we write $\log(x)\equiv \log(e\vee x)$, and $\log \log (x)\equiv \log(e \vee \log (e\vee x))$.

\section{Empirical processes indexed by sets}\label{section:ep_sharp_bounds}

\subsection{Setup and assumptions}

Let $X_1,\ldots,X_n$ be i.i.d. random variables with distribution $P$ on a sample space $(\mathcal{X}, \mathcal{A})$, and $\mathscr{C}$ be a collection of measurable sets contained in $\mathcal{X}$. To avoid measurability digressions, we assume that $\mathscr{C}$ is countable throughout the article. For any $\sigma>0$, let $\mathscr{C}(\sigma)\equiv\{C \in \mathscr{C}: P(C)\leq \sigma^2\}$. 

Following the standard notation for covering and bracketing numbers (cf. \cite[pp. 83]{van1996weak}), for a normed linear space $(\mathcal{F},\pnorm{\cdot}{})$, let the covering number  $\mathcal{N}(\epsilon, \mathcal{F},\pnorm{\cdot}{})$ be the minimum number of balls $\{g: \pnorm{g-f}{}<\epsilon\}$ of radius $\epsilon$ under $\pnorm{\cdot}{}$ needed to cover $\mathcal{F}$. Let the bracketing number $\mathcal{N}_{[\,]}(\epsilon, \mathcal{F},\pnorm{\cdot}{})$ be the minimum number of $\epsilon$-brackets under $\pnorm{\cdot}{}$ needed to cover $\mathcal{F}$, where an $\epsilon$-bracket $[\ell,u]\equiv \{f: \ell\leq f\leq u\}$ with $\pnorm{u-\ell}{}<\epsilon$. Following the notation in \cite[pp. 270, (7.4)]{dudley1999uniform}, let $\mathcal{N}_I(\epsilon,\mathscr{C},P)$ be the $\epsilon$-bracketing number for $\mathscr{C}$ under $P$, i.e. $\mathcal{N}_I(\epsilon,\mathscr{C},P)\equiv \mathcal{N}_{[\,]}(\epsilon, \bm{1}_{\mathscr{C}}, L_1(P))$ with $\bm{1}_{\mathscr{C}}\equiv \{\bm{1}_C:C\in \mathscr{C}\}$.

\begin{assumption}
	Fix $\alpha>0$.
	\begin{enumerate}
		\item[(E1)] $\log \mathcal{N}_I(\epsilon,\mathscr{C},P)\leq L\epsilon^{-\alpha}$.
		\item[(E2)] $\log \mathcal{N}(\epsilon/4,\mathscr{C}(\sqrt{\epsilon}),P)\geq L^{-1}\epsilon^{-\alpha}.$
	\end{enumerate}
\end{assumption}
For examples satisfying the above entropy conditions, see e.g. \cite[Sections 8.3/8.4, or Theorem 8.3.2]{dudley1999uniform} on the class of upper/lower sets and convex bodies (cf. \cite[Theorem 8.4.1/Corollary 8.4.2]{dudley1999uniform}). $L$ will be a large enough absolute constant throughout the article, the dependence on which will not be explicitly stated in the theorems.

For $0<\alpha<1$, the bracketing condition in (E1) can also be replaced by a uniform entropy condition $
\sup_{Q} \log \mathcal{N}(\epsilon, \mathscr{C},Q)\leq L\epsilon^{-\alpha}$, 
where the supremum is taken over all finitely discrete probability measures $Q$. Such a uniform entropy condition is satisfied if $\mathscr{C}$ is a VC-class (cf. \cite[Section 3.6]{gine2015mathematical}).\footnote{\cite{baraud2016bounding} advocates the notion of weak VC-major class as a generalization of VC-major class.  The class of indicators over a class of sets $\mathscr{C}$ is weakly VC-major if and only if $\mathscr{C}$ is VC (cf. \cite[Definition 2.2]{baraud2016bounding}). } This case is essentially covered in \cite{gine2006concentration}.

\subsection{Upper and lower bounds}

We first state the general upper and lower bounds for empirical processes indexed by general function classes satisfying certain entropy conditions.

\begin{theorem}\label{thm:upper_bound_ep}
	
	\begin{enumerate}
		\item Fix $p\geq 1$. Let $\mathcal{F}\subset L_\infty(1)$, and $\sup_{f \in \mathcal{F}} \pnorm{f}{L_p}\leq \sigma$. Suppose there exists some $\alpha> 0$ such that for all $\epsilon>0$, $
		\log \mathcal{N}_{[\,]}(\epsilon,\mathcal{F},L_p(P))\leq L\epsilon^{-\alpha}$. Then 
		\begin{align}\label{ineq:upper_bound_ep_abstract}
		&\E \sup_{f \in \mathcal{F}} \abs{\G_n(f)}  \nonumber\\
		&\lesssim_{\alpha,p} \inf_{0\leq \gamma \leq \sigma/2 } \bigg\{(\sigma\bm{1}_{\alpha<p\wedge 2}+\gamma\bm{1}_{\alpha>p\wedge 2})^{ \{ (p\wedge 2)-\alpha\}/2}+ n^{-1/2} \sigma^{-\alpha} + \sqrt{n} \gamma \bigg\} \nonumber \\
		&\asymp 
		\begin{cases}
		\sigma^{\{ (p\wedge 2)-\alpha\}/2} + n^{-1/2}\sigma^{-\alpha}, & \alpha<p\wedge 2;\\
		n^{ \frac{\alpha-p\wedge 2}{2(\alpha+2-p\wedge 2)} } + \sigma^{-\frac{\alpha-p\wedge 2}{2}}+n^{-1/2} \sigma^{-\alpha}, & \alpha> p\wedge 2.
		\end{cases} 
		\end{align}
		\item 
		Suppose that the following entropy estimate holds for some $\alpha>0$:  \begin{align}\label{ineq:upper_bound_ep_1}
			\log \mathcal{N}_{[\,]}(\epsilon^2,\mathcal{F},L_1(P)\bm{1}_{\alpha>1}+L_2^2(P) \bm{1}_{0<\alpha<1})\leq L\epsilon^{-2\alpha}.
			\end{align}
			Then for $\sigma^2\gtrsim n^{-1/(\alpha+1)}$, $\alpha\neq 1$, we have: 
			\begin{align}\label{ineq:upper_bound_ep}
			\E \sup_{f \in \mathcal{F}(\sigma)}\abs{\G_n(f)}\lesssim_{\alpha} \max\big\{ \sigma^{1-\alpha}, n^{(\alpha-1)/2(\alpha+1)}\big\}.
			\end{align}
   \item 
	If in addition to (\ref{ineq:upper_bound_ep_1}), it holds that
	\begin{align}\label{ineq:thm_lower_bound_1}
	 \log \mathcal{N}(\epsilon/2,\mathcal{F}(\epsilon),L_2(P) )\geq L^{-1}\epsilon^{-2\alpha}.
	\end{align}
	Then for $\sigma^2\gtrsim n^{-1/(\alpha+1)}$, $\alpha\neq 1$, we have: 
	\begin{align}\label{ineq:lower_bound_ep}
	\E \sup_{f \in \mathcal{F}(\sigma)}\abs{\G_n(f)}\gtrsim_{\alpha} \max\big\{ \sigma^{1-\alpha}, n^{(\alpha-1)/2(\alpha+1)}\big\}.
	\end{align}
\end{enumerate}
Here $\mathcal{F}(\sigma)\equiv \{f \in \mathcal{F}: Pf^2\leq \sigma^2\}$.
\end{theorem}
\begin{proof}
See Section \ref{section:proof_ep}.
\end{proof}

\begin{remark}[Comparison to classical bounds]\label{rmk:classical_comparison}
By the standard local maximal inequality for the empirical processes (cf. \cite[Lemma 2.14.3]{van1996weak}), we have for $\sigma>0$ not too small,
\begin{align}\label{ineq:classical_local_maximal_ineq}
\E \sup_{f \in \mathcal{F}(\sigma)} \abs{\G_n(f)}&\lesssim \inf_{ 0\leq \gamma\leq \sigma/2 }\left\{ \sqrt{n}\gamma+\int_{\gamma}^{\sigma}\sqrt{\log \mathcal{N}_{[\,]}(\epsilon,\mathcal{F},L_2(P))}\ \d{\epsilon}\right\}.
\end{align}
	Suppose $\mathcal{F}\subset L_\infty(1)$ satisfies the $L_2$ entropy condition:
	\begin{align}\label{ineq:L2_entropy}
	\epsilon^{-2\alpha}\lesssim \log \mathcal{N}(\epsilon,\mathcal{F}(2\epsilon), L_2(P))\leq \log \mathcal{N}_{[\,]}(\epsilon, \mathcal{F}, L_2(P))\lesssim \epsilon^{-2\alpha}.
	\end{align}
	In the Donsker regime $\alpha<1$, standard upper bound (\ref{ineq:classical_local_maximal_ineq}) and the lower bound (\ref{ineq:lower_bound_ep}) already match each other under (\ref{ineq:L2_entropy}). In the non-Donsker regime $\alpha>1$, these bounds lead to (\ref{ineq:ep_size_classic_nondonsker}) whose upper and lower bounds do not match, both of which can be attained for certain instances of $\mathcal{F}$ satisfying (\ref{ineq:L2_entropy}). This means improvements for the generic bound (\ref{ineq:ep_size_classic_nondonsker}) must require additional structural assumptions on $\mathcal{F}$. As the $L_1$ entropy condition (\ref{ineq:upper_bound_ep_1}) along with $\mathcal{F}\subset L_\infty(1)$ implies the rightmost part of (\ref{ineq:L2_entropy}), it follows that within the family of $\mathcal{F}\subset L_\infty(1)$'s satisfying the $L_2$ entropy condition (\ref{ineq:L2_entropy}) with $\alpha>1$, if in addition the stronger $L_1$ entropy condition (\ref{ineq:upper_bound_ep_1}) is satisfied, then the upper bound in (\ref{ineq:ep_size_classic_nondonsker}) can be improved from $n^{(\alpha-1)/2\alpha}$ to $n^{(\alpha-1)/2(\alpha+1)}$ that matches the lower bound. 
\end{remark}

An important family of $\mathcal{F}$ satisfying the $L_1$ entropy condition (\ref{ineq:upper_bound_ep_1}) is the class of indicators over measurable sets. We formalize the results below.

\begin{theorem}\label{thm:upper_bound_ep_set}
	Suppose (E1) holds and $\sigma^2\gtrsim n^{-1/(\alpha+1)}$, $\alpha\neq 1$. Then
	\begin{align*}
	\E \sup_{C \in \mathscr{C}(\sigma)}\abs{\G_n(C)} \lesssim_{\alpha} \max\big\{\sigma^{1-\alpha}, n^{(\alpha-1)/2(\alpha+1)}\}.
	\end{align*}
	If furthermore (E2) holds, then
	\begin{align*}
	\E \sup_{C \in \mathscr{C}(\sigma)}\abs{\G_n(C)} \gtrsim_{\alpha} \max\big\{\sigma^{1-\alpha}, n^{(\alpha-1)/2(\alpha+1)}\}.
	\end{align*}
\end{theorem}

\begin{proof}[Proof of Theorem \ref{thm:upper_bound_ep_set}]
	(\ref{ineq:upper_bound_ep_1}) is verified using the fact that for any measurable set $C$, with $f\equiv \bm{1}_C$ we have $Pf = Pf^2$. (\ref{ineq:thm_lower_bound_1}) can be verified by noting that $\mathcal{N}_I (\epsilon/4, \mathscr{C}(\sqrt{\epsilon'}), P) = \mathcal{N}_{[\,]}(\sqrt{\epsilon}/2, \mathcal{F}(\sqrt{\epsilon'}), L_2(P))$ holds with $\mathcal{F}\equiv \{\bm{1}_C: C \in \mathscr{C}\}$ and any $\epsilon,\epsilon'>0$ (and similarly for the covering number).
\end{proof}

\begin{remark}\label{rmk:upper_bound_entropy}
Some remarks on the upper bounds in Theorems \ref{thm:upper_bound_ep} and \ref{thm:upper_bound_ep_set}:
\begin{enumerate}

    \item Roughly speaking, the improved estimates (\ref{ineq:upper_bound_ep_abstract}) compared to the classical bound (\ref{ineq:classical_local_maximal_ineq}) come from careful $L_p$ chaining with bracketing that provides tighter controls at the finest resolution of the chaining step; see Section \ref{section:proof_upper_bounds} for some heuristics and details.
	\item It is also possible to consider the boundary case $\alpha=1$ in Theorem \ref{thm:upper_bound_ep_set}. Then the upper bound deviates from the lower bound by a logarithmic factor. In particular, suppose (E1)-(E2) hold and $\sigma^2\gtrsim n^{-1/2}$. Then $
	1\lesssim \E \sup_{C \in \mathscr{C}(\sigma)} \abs{\G_n(C)} \lesssim \log n$.
\end{enumerate}
\end{remark}

\begin{remark}\label{rmk:lower_bound_entropy}
Some remarks on the lower bounds in Theorems \ref{thm:upper_bound_ep} and \ref{thm:upper_bound_ep_set}:
	\begin{enumerate}
		\item 
		The proof for the lower bound (\ref{ineq:lower_bound_ep}) is based on Gaussian randomization followed by an application of the multiplier inequality derived in the author's previous work \cite{han2017sharp} that removes the effect of Gaussianization. This only requires some sharp upper bounds for the unconditional processes, as opposed to the approach of \cite{gine2006concentration} using Rademacher minorization, which requires sharp upper bounds for \emph{conditional} processes. 
	   \item Condition (\ref{ineq:thm_lower_bound_1}) is also assumed  in \cite{gine2006concentration} under the name `$\alpha$-fullness' (cf. Definition 3.3 therein). This condition is best verified on a case-by-case basis. For instance the $\alpha$-H\"older class on $[0,1]$ is $\alpha$-full, cf. the proof of Lemma 6 in \cite{han2017sharp}. 
	\end{enumerate}
\end{remark}

\section{Rate-optimal global ERMs}\label{section:application}

In this section, we apply the new bounds derived in the previous section to several models including (i) image and edge estimation, (ii) binary classification, (iii) multiple isotonic regression, and (iv) $s$-concave density estimation, all in general dimensions. Global ERMs in these models are non-Donsker problems in general dimensions, but we will show that in each of these models, the underlying empirical process problem (\ref{ineq:risk_characterization}) can be essentially characterized by an empirical process indexed by certain class of measurable sets. The bounds in Theorem \ref{thm:upper_bound_ep_set} can then be used to prove that these global ERMs converge at an optimal rate (up to logarithmic factors), rather than a strictly sub-optimal rate as predicted using the entropy integral (= (\ref{ineq:upper_bound_risk})) in \cite{birge1993rates}. Interestingly, Theorem \ref{thm:upper_bound_ep_set} is typically applied without the localization. This is viable as the size of the expected supremum of empirical process is already diverging in the non-Donsker regime, so localization is usually uninformative.

\subsection{Image estimation}

Let $X_1,\ldots,X_n$ be i.i.d. samples with law $P$ on a sample space $(\mathcal{X}, \mathcal{A})$. In this subsection we consider the regression model:
\begin{align}\label{eqn:reg_model}
Y_i = \bm{1}_{C_0}(X_i)+\xi_i,\quad i=1,\ldots,n.
\end{align}
This model has been considered by \cite{korostelev1992asymtotically,korestelev1993minimax} and more recently by \cite{brunel2013adaptive} under the name `image estimation', cf. \cite[Section 3.1]{korestelev1993minimax}, where $C_0$ is considered as the `image', and $\mathcal{X}\setminus C_0$ is considered as `background'. We assume for simplicity that the $\xi_i$'s are i.i.d. $\mathcal{N}(0,1)$ and are independent of $X_i$'s. Let $\mathscr{C}$ be a collection of measurable sets in $\mathcal{X}$, and we will fit the regression model by $\{\bm{1}_C: C \in \mathscr{C}\}$. Our interest will be the behavior of the least squares estimator $\widehat{C}_n$ defined by
\begin{align}\label{def:lse_reg_add}
\widehat{C}_n \in \argmin_{C \in \mathscr{C}} \sum_{i=1}^n\left(Y_i-\bm{1}_C(X_i)\right)^2.
\end{align}
We assume that $\widehat{C}_n$ is well-defined without loss of generality. We will measure the quality of $\widehat{C}_n$ via the expected symmetric difference of $\widehat{C}_n$ and $C_0$ under $P$ defined by
\begin{align}\label{def:symmetric_diff}
P\abs{\widehat{C}_n\Delta C_0} = P\big(\bm{1}_{\widehat{C}_n}-\bm{1}_{C_0}\big)^2 = \int \big(\bm{1}_{\widehat{C}_n}-\bm{1}_{C_0}\big)^2\ \d{P}.
\end{align}
By a relatively standard reduction (cf. Lemma \ref{lem:lse_additive_error}), the risk of $\widehat{C}_n$ in symmetric difference can be related to the expected supremum
\begin{align*}
	\E \sup_{C \in \mathscr{C}: P\abs{C\Delta C_0}\leq \delta^2} \biggabs{\frac{1}{\sqrt{n}}\sum_{i=1}^n \epsilon_i(\bm{1}_{C}-\bm{1}_{C_0})(X_i) },
\end{align*}
and therefore we may apply Theorem \ref{thm:upper_bound_ep_set} (or the more general Theorem \ref{thm:upper_bound_ep}). We formalize the result below.

\begin{theorem}\label{thm:rate_optimal_lse_reg}
	Suppose that for some $\alpha\neq 1$, $
	\log \mathcal{N}_I(\epsilon,\mathscr{C},P)\leq L\epsilon^{-\alpha}$. Then
	\begin{align*}
	\sup_{C_0 \in \mathscr{C}} \E_{C_0} P\abs{\widehat{C}_n\Delta C_0} \lesssim n^{-1/(\alpha+1)}.
	\end{align*}
\end{theorem}
\begin{proof}
See Section \ref{section:proof_application}.
\end{proof}

By \cite{yang1999information}, the rate $n^{-1/(\alpha+1)}$ cannot be improved in a minimax sense if furthermore a lower bound on the metric entropy on the same order as that of the upper bound is available.

As a straightforward corollary of the above Theorem \ref{thm:rate_optimal_lse_reg}, let $\mathscr{C}_d$ be the collection of all convex bodies contained in the unit ball in $\R^d$ and $P$ the uniform distribution on the unit ball.

\begin{corollary}
	Fix $d\geq 4$. Then
	\begin{align*}
	\sup_{C_0 \in \mathscr{C}_d} \E_{C_0} P\abs{\widehat{C}_n\Delta C_0} \lesssim n^{-2/(d+1)}.
	\end{align*}
\end{corollary}
\begin{proof}
	The claim essentially follows from \cite[Theorem 8.25, Corollary 8.26 ]{dudley1999uniform}, asserting that we can take $\alpha=(d-1)/2$ in Theorem \ref{thm:rate_optimal_lse_reg}.
\end{proof}

The corollary shows that we can use a global least squares estimator rather than a sieved least squares estimator (cf. \cite{brunel2013adaptive}) to achieve the optimal rate of convergence.

\begin{remark}
It is possible to impose certain tail conditions on the density of $P$ to extend the above corollary to a maximum risk bound over all convex sets in $\R^d$. In particular, the above result holds for any $P$ with compact support in $\R^d$ with a bounded Lebesgue density. A proof in this vein is carried out in the context of $s$-concave density estimation in $\R^d$ to be detailed ahead.
\end{remark}

\subsection{Edge estimation}
In this subsection we consider the regression model studied in \cite{korestelev1993minimax,mammen1995asymptotical}:
\begin{align}\label{eqn:reg_multi_model_original}
Y_i = f_{C_0}(X_i)\eta_i
\end{align}
where $f_{C_0}(x)=2\bm{1}_{C_0}(x)-1$ and $\eta_i$'s are i.i.d. random variables such that $\Prob(\eta_i=1)=1/2+a$ and $\Prob(\eta_i=-1)=1/2-a$ for some known constant $a \in (0,1/2)$. Such a model is motivated by estimation of sets in multi-dimensional `black and white' pictures, where $Y_i=1$ is interpreted as observing black, and $Y_i=-1$ is white. We refer the reader to \cite{mammen1995asymptotical} for more motivation for this model. The model (\ref{eqn:reg_multi_model_original}) can be rewritten as
\begin{align}\label{eqn:reg_multi_model}
Y_i = 2af_{C_0}(X_i)+\xi_i
\end{align}
where $\xi_i = f_{C_0}(X_i)(\eta_i-2a)$'s are bounded errors. An important property for these errors is that $\E[\xi_i|X_i]=0$ for all $i=1,\ldots,n$. Note here $\xi_i$ is \emph{not} independent of $X_i$ and hence a different analysis is needed. Now consider the least squares estimator
\begin{align}\label{def:lse_reg_multi}
\widehat{C}_n\equiv \argmin_{C \in \mathscr{C}}\sum_{i=1}^n(Y_i-2af_C(X_i))^2.
\end{align}
A careful analysis to be detailed in Lemma \ref{lem:lse_dependent_bounded_error} ahead shows that the risk of $\widehat{C}_n$ in symmetric difference (\ref{def:symmetric_diff}) can still be related to the expected supremum of empirical process, so Theorem \ref{thm:upper_bound_ep_set} is applicable in this setting as well. Formally, we have
\begin{theorem}\label{thm:tsybakov_example}
	Suppose that for some $\alpha\neq 1$, $
	\log \mathcal{N}_I(\epsilon,\mathscr{C},P)\leq L\epsilon^{-\alpha}$. 
	Then
	\begin{align*}
	\sup_{C_0 \in \mathscr{C}} \E_{C_0} P\abs{\widehat{C}_n\Delta C_0} \lesssim n^{-1/(\alpha+1)}.
	\end{align*}
\end{theorem}
\begin{proof}
	See Section \ref{section:proof_application}.
\end{proof}

Compared to \cite[Theorem 4.1]{mammen1995asymptotical}, we use an unsieved least squares estimator to achieve the optimal rate, rather than their theoretical `sieved' estimator. This provides another example for which the simple least squares estimator can be rate-optimal for non-Donsker function classes in a natural setting.

\subsection{Binary classification: excess risk bounds}\label{section:binary_classification}

In this subsection we consider the binary classification problem in the learning theory, cf. \cite{tsybakov2004optimal,massart2006risk}. Suppose one observes i.i.d. $(X_1,Y_1),\ldots,(X_n,Y_n)$ with law $P$, where $X_i$'s take values in $\mathcal{X}$, and the responses $Y_i \in \{0,1\}$. A classifier $g: \mathcal{X} \to \{0,1\}$ over a class $\mathcal{G}$ has a generalization error $P(Y\neq g(X))$. The excess risk for a classifier $g$ over $\mathcal{G}$ under law $P$ is given by
\begin{align*}
\mathcal{E}_P(g)\equiv P(Y\neq g(X))-\inf_{g' \in \mathcal{G}} P(Y\neq g'(X)).
\end{align*}
It is known that for a given law $P$ on $(X,Y)$, the minimal generalized error is attained by a Bayes classifier $g_0(x)\equiv \bm{1}_{\eta(x)\geq 1/2}$ where $\eta(x)\equiv \E[Y|X=x]$, cf. \cite{devroye1996probabilistic}. It is then natural to consider an estimator of $g_0$ by minimizing the empirical training error:
\begin{align}\label{def:ERM_classification}
\widehat{g}_n \equiv \argmin_{g \in \mathcal{G}} \frac{1}{n}\sum_{i=1}^n \bm{1}_{Y_i\neq g(X_i)}.
\end{align}
We assume $g_0 \in \mathcal{G}$ for simplicity. The global ERM $\widehat{g}_n$ is previously studied in, e.g., \cite[pp. 136]{tsybakov2004optimal}, \cite[pp. 2327]{massart2006risk}, \cite[pp. 2627-2629]{koltchinskii2006local}, \cite[pp. 1211-1213]{gine2006concentration}. The quality of the estimator $\widehat{g}_n$ is measured by the excess risk:
\begin{align*}
\mathcal{E}_P(\widehat{g}_n)\equiv P(Y\neq \widehat{g}_n(X))-P(Y\neq g_0(X)).
\end{align*}
Let $\Pi$ be the marginal distribution of $X$ under $P$. We assume the following `Tsybakov's margin(low noise) condition' (cf. \cite{mammen1999smooth,tsybakov2004optimal}): there exists some $c>0$ such that for all $g \in \mathcal{G}$,
\begin{align}\label{ineq:margin_cond}
\mathcal{E}_P(g)\geq c\big(\Pi(g(X)\neq g_0(X))\big)=c \pnorm{g-g_0}{L_2(P)}^{2}.
\end{align}
Here we have assumed that the margin condition holds with $\kappa=1$. Although faster rates are possible under more general margin condition $\kappa\geq 1$ (cf. \cite{mammen1999smooth,tsybakov2004optimal}),  we do not go into this direction to avoid distraction from our main points. 

Below is the main result in this subsection, the formulation of which follows that of \cite{koltchinskii2006local,gine2006concentration}.

\begin{theorem}\label{thm:excess_risk_binary}
	Suppose $\mathcal{G}\equiv \{\bm{1}_C:C \in \mathscr{C}\}$ satisfies the following entropy condition: there exists some $\alpha\neq 1$ such that for all $\epsilon>0$, $
	\log \mathcal{N}_I(\epsilon, \mathscr{C},P)\leq L\epsilon^{-\alpha}$. If $r_n^2\geq K n^{-1/(\alpha+1)}$ for a large enough constant $K>0$, then
	\begin{align*}
	\Prob\left(\mathcal{E}_P(\widehat{g}_n)\geq r_n^2\right)\leq K'\exp(-nr_n^2/K')
	\end{align*}
	holds for some constant $K'>0$.
\end{theorem}
\begin{proof}
	See Section \ref{section:proof_application}.
\end{proof}

Roughly speaking, the key to prove the above theorem is a control for the random variable
\begin{align*}
\max_{1\leq j\leq \ell} \frac{\sup_{f \in \mathcal{F}_j} \abs{\Prob_n(f)-P(f)}  }{r_n^2 2^j}
\end{align*}
where $\ell$ is the smallest integer such that $r_n^2 2^\ell\geq 1$, and $\mathcal{F}_j\equiv \{\bm{1}_{y\neq g_1(x)}-\bm{1}_{y\neq g_2(x)}: \mathcal{E}_P(g_1)\vee \mathcal{E}_P(g_2)\leq r_n^2 2^j\}$. A sharp estimate for the above random variable is achieved by an application of Theorem \ref{thm:upper_bound_ep_set} and Talagrand's inequality (cf. Appendix \ref{section:tools}).

Examples of $\mathcal{G}$ that satisfy the prescribed entropy conditions in the above theorem can be found in the comments after \cite[Theorem 1, pp. 1813]{mammen1999smooth}. To put the above results in the literature, \cite{tsybakov2004optimal} considered the same problem under the working assumption $\alpha \in (0,1)$ (cf. \cite[Assumption A2, pp. 140]{tsybakov2004optimal}). \cite{massart2006risk} used ratio-type empirical process techniques to give a more unified treatment of deriving risk bounds for this problem, when the class of classifiers satisfies a Donsker bracketing entropy condition (i.e. $0<\alpha<1$), or a Donsker uniform entropy condition. \cite{gine2006concentration} further improved the result of \cite{massart2006risk} in the Donsker regime under a uniform entropy condition, by taking into account the size of the localized envelopes. See also \cite[pp. 2618]{koltchinskii2006local}, \cite[pp. 1706]{lecue2007simultaneous} for similar Donsker conditions. To the best knowledge of the author, our Theorem \ref{thm:excess_risk_binary} gives a first result for the global ERM $\widehat{g}_n$ in (\ref{def:ERM_classification}) to be rate-optimal in the non-Donsker regime $\alpha>1$ in the classification problem.

\subsection{Multiple isotonic regression}

Let $X_1,\ldots,X_n$ be i.i.d. with law $P$ on $[0,1]^d$. For simplicity we assume that $P$ is the uniform distribution on $[0,1]^d$. Consider the multiple isotonic regression model
\begin{align}\label{eqn:reg_model_shape}
Y_i = f_0(X_i)+\xi_i,\quad i=1,\ldots,n,
\end{align}
where $\xi_i$'s are i.i.d. Gaussian errors $\mathcal{N}(0,1)$, and  $f_0 \in \mathcal{M}_d \equiv \{f: [0,1]^d \to \R, f(x)\leq f(y)\textrm{ for any }x\leq y\}$. Consider the isotonic least squares regression estimator $\widehat{f}_n$ defined via:
\begin{align*}
\widehat{f}_n \equiv \argmin_{f \in \mathcal{M}_d} \sum_{i=1}^n \big(Y_i-f(X_i)\big)^2.
\end{align*}
The performance of $\widehat{f}_n$ in the multivariate setting is examined by \cite{chatterjee2018matrix} for $d=2$ and \cite{han2017isotonic} for $d\geq 3$. By the entropy estimate for uniformly bounded multiple isotonic functions in \cite{gao2007entropy}, $\mathcal{M}_d \cap L_\infty(1)$ is in the non-Donsker regime when $d\geq 3$, which is the main interesting case here.

\begin{theorem}\label{thm:iso_reg}
Let $d\geq 2$. Then with $\gamma_{d,\textrm{iso}} =(2)\bm{1}_{d=2}+\bm{1}_{d\geq 3}$, 
\begin{align*}
\sup_{f_0 \in \mathcal{M}_d \cap L_\infty(1)} \E_{f_0} \pnorm{\widehat{f}_n-f_0}{L_2(P)}^2\lesssim_d n^{-1/d} \log^{\gamma_{d,\textrm{iso}} } n.
\end{align*}
\end{theorem}
\begin{proof}
	See Section \ref{section:proof_application}.
\end{proof}

Compared to \cite[Theorem 4]{han2017isotonic}, the above result gives improvements over logarithmic factors. The logarithmic factors in $d\geq 3$ are due to boundary behavior of $\widehat{f}_n$. For instance, if the errors are bounded, then we may remove these logarithmic factors to get a sharp rate $n^{-1/d}$ for $d\geq 3$. These logarithmic factors cannot be removed by the proof techniques in \cite{han2017isotonic} even if the errors are bounded. The rate $n^{-1/d}$ is shown to be minimax optimal for squared $L_2$ loss in \cite{han2017isotonic}.

	The proof of Theorem \ref{thm:iso_reg} contains two inter-related steps:
\begin{enumerate}
	\item First, we show that with high enough probability,
	\begin{align*}
	\pnorm{\widehat{f}_n-f_0}{\infty} = \mathcal{O}(\sqrt{\log n})
	\end{align*}
	under the Gaussian noise assumption.
	\item Second, using the first step, we show that
	\begin{align*}
	\E \sup_{f \in \mathcal{M}_d \cap L_\infty(C\sqrt{\log n})} \abs{\G_n(f-f_0)}\lesssim \sqrt{\log n} \cdot \E \sup_{C \in \mathcal{L}_d} \abs{\G_n(C)},
	\end{align*}
	where $\mathcal{L}_d$ is the collection of all upper and lower sets contained in $[0,1]^d$ (precise definition see the paragraph before the proof of Theorem \ref{thm:iso_reg} in Section \ref{section:proof_application}). The expected supremum on the right hand side of the above display can be controlled using Theorem \ref{thm:upper_bound_ep_set}. Finally the claim follows by a standard reduction for the risk of LSE to expected supremum of empirical process.
\end{enumerate}

 Compared to the proof of \cite[Theorem 4]{han2017isotonic}, the proof strategy described above is more informative by making a clear connection to the class $\mathcal{L}_d$ that drives the minimax rates of convergence for the multiple isotonic LSE. See also Remark \ref{rmk:isotonic_reg} for a detailed technical comparison.

The approach outlined above can also be adapted to the problem of multivariate convex regression modulo technical difficulties due to unsolved boundary behavior of the convex LSE. See Remark \ref{rmk:conv_reg} for some details.

\subsection{$s$-concave density estimation in $\R^d$}

We first introduce the class of $s$-concave densities on $\R^d$. The exposition follows that of \cite{han2015approximation}. Let
\begin{align*}
M_s(a,b;\theta)\equiv
\begin{cases}
\big((1-\theta)a^s+\theta b^s\big)^{1/s}, & s\neq 0, a,b > 0,\\
0, & s <0, ab = 0,\\
a^{1-\theta}b^\theta, &s=0,\\
a\wedge b, &s=-\infty.
\end{cases}
\end{align*}
A density $p$ on $\R^d$ is called $s$-concave, i.e. $p \in \mathcal{P}_s$ 
if and only if for all $x_0,x_1\in \R^d$ and $\theta \in (0,1)$,
$p\big((1-\theta)x_0+\theta x_1\big)\geq M_s(p(x_0),p(x_1);\theta)$. It is easy to see that the densities $p$ have the form $p=\varphi_+^{1/s}$ 
for some concave function $\varphi$ if $s>0$, $p=\exp(\varphi)$ for some concave 
$\varphi$ if $s=0$, and $p=\varphi_+^{1/s}$ for some convex $\varphi$ if $s<0$. 
The function classes $\mathcal{P}_s$ are nested in $s$ in that for every $r>0>s$, we have
$\mathcal{P}_r\subset \mathcal{P}_0\subset \mathcal{P}_s\subset \mathcal{P}_{-\infty}.$

Maximum likelihood estimation over $\mathcal{P}_s$ is proposed in \cite{seregin2010nonparametric}, where existence and consistency of the MLE $\widehat{p}_n$ is proved. Global rates of convergence of the MLE $\widehat{p}_n$ over $\mathcal{P}_s$ is primarily studied in the special case $s=0$, also known as the log-concave MLE, cf. \cite{kim2016global}. For general $s$-concave MLEs, the only result concerning global convergence rates is due to \cite{doss2013global}, who studied the univariate case $d=1,s>-1$, showing that $h^2(\widehat{p}_n,p_0) = \mathcal{O}_{\mathbf{P}}(n^{-4/5})$, where $h(\cdot,\cdot)$ is the Hellinger distance. Here we will be interested in general $s$-concave MLEs in general dimensions. 

\begin{theorem}\label{thm:rate_s_concave}
Suppose $s>-1/d$ and $d\geq 2$. Then 
\begin{align*}
h^2(\widehat{p}_n,p_0) = \mathcal{O}_{\mathbf{P}}(n^{-2/(d+1)}\log^{\gamma_{d,s}} n),
\end{align*}
where $
\gamma_{d,s} = (2/3)\bm{1}_{d=2}+(2)\bm{1}_{d=3}+\bm{1}_{d\geq 4}$.
\end{theorem}
\begin{proof}
	See Section \ref{section:proof_application}.
\end{proof}

The most interesting regime here is $d\geq 4$ when the entropy integral for the class of $s$-concave densities diverges. Modulo logarithmic factors, the rates of convergence for the $s$-concave MLE $\widehat{p}_n$ in squared Hellinger distance is $\mathcal{O}_{\mathbf{P}}(n^{-2/(d+1)})$, which matches the minimax lower bound for the smaller log-concave (= 0-concave) class, cf. \cite{kim2016global}. During the preparation the paper, the author becomes aware of the very nice work \cite{dagan2019log} which derives, among other things, global risk bounds for the log-concave (i.e. $s=0$) MLEs in the Hellinger distance. The techniques used in both papers in this example share certain common features, while our general setting brings about further technical challenges. See Remark \ref{rmk:s_concave_reduction} below for more technical comments on the proof of Theorem \ref{thm:rate_s_concave}.

The integrability restriction $s>-1/d$ is very natural in this setting: if $s<-1/d$, then there exists a family of $s$-concave densities with singularities so that the MLE does not exist. The following proposition makes this precise.

\begin{proposition}\label{prop:s_concave}
The $s$-concave MLE does not exist for $s<-1/d$.
\end{proposition}
\begin{proof}
For $a\in \R^d, b >0$, let $\widetilde{\varphi}_{a,b}(x)\equiv \pnorm{x-a}{}\bm{1}_{\pnorm{x-a}{}\leq b}+\infty\bm{1}_{\pnorm{x-a}{}>b}$. Since $c_b\equiv \int \widetilde{\varphi}_{a,b}^{1/s} = \int \widetilde{\varphi}_{0,b}^{1/s}<\infty$ for $s<-1/d$, $p_{a,b}\equiv \widetilde{\varphi}_{a,b}^{1/s}/ c_b$ is an $s$-concave density. The log likelihood function for observed $X_1,\ldots,X_n$ is $\ell(a,b)\equiv\log \prod_{i=1}^n p_{a,b}(X_i) =\sum_{i=1}^n \big[(1/s)\log(\pnorm{X_i-a}{}) - \log c_b\big] $ for $(a,b)$ such that $\max_i \pnorm{X_i-a}{}\leq b$ and $X_i\neq a$ for $i=1,\ldots,n$. For $b$ large enough and $a$ approaches any of $X_i$'s, $\ell(a,b)\uparrow \infty$, so the MLE does not exist.
\end{proof}

The univariate case $d=1$ for the above proposition can also be found in \cite{doss2013global}. 

	\begin{remark}\label{rmk:s_concave_reduction}
		The proof of Theorem \ref{thm:rate_s_concave} relies on the following reduction scheme:
		\begin{align}\label{ineq:s_concave_reduction}
		h^2(\widehat{p}_n,p_0)&\lesssim \log n\cdot \E \sup_{C \in \mathscr{C}_d} \abs{(\Prob_n-P_0) (C)}+ \mathcal{O}_{\mathbf{P}}(n^{-1/2}),
		\end{align}
		where $\mathscr{C}_d$ is the class of convex bodies on $\R^d$. 
		
		The above reduction scheme (\ref{ineq:s_concave_reduction}) for $s=0$ is essentially achieved in \cite{carpenter2018near}, but a sharp bound for the expected supremum of the empirical process on the right hand side of the above display is not available therein. Here we show that the reduction (\ref{ineq:s_concave_reduction}) holds for the maximum regime in which the $s$-concave MLE exists. Once (\ref{ineq:s_concave_reduction}) is proven, the expected supremum on its right hand side can be controlled by Theorem \ref{thm:upper_bound_ep_set} combined with a standard technique of `domain extension' (cf. \cite{van1996new} or \cite[Corollary 2.7.4]{van1996weak}) under the envelope control (\ref{ineq:s_concave_0}). See the proof of Theorem \ref{thm:rate_s_concave} for more details. 
	\end{remark}

\section{Proofs for Section \ref{section:ep_sharp_bounds}}\label{section:proof_ep}

We will prove Theorem \ref{thm:upper_bound_ep} in this section. As (2) is a direct consequence of (1), we will prove (1) and (3) only. 

\subsection{Proof of Theorem \ref{thm:upper_bound_ep}-(1)}\label{section:proof_upper_bounds}

A heuristic way of seeing the bound (\ref{ineq:upper_bound_ep_abstract}) is the following. This requires some understanding for the proof of the classical maximal inequality (\ref{ineq:classical_local_maximal_ineq}):
\begin{itemize}
	\item The entropy integral term $\int_{\gamma}^{\sigma} \sqrt{\log \mathcal{N}_{[\,]}(\epsilon, \mathcal{F}, L_2(P)}\,\d{\epsilon}$ in (\ref{ineq:classical_local_maximal_ineq}) comes from $L_2$ chaining with bracketing in the `Gaussian regime' using Bernstein's inequality, starting from $\sigma$ down to $\gamma$. 
	\item The residual term $\sqrt{n}\gamma$ in (\ref{ineq:classical_local_maximal_ineq}) comes from the bound $\pnorm{f}{L_1(P)}\leq \pnorm{f}{L_2(P)}$ towards the end level $\gamma$ of the $L_2$ chaining. 
\end{itemize}
Now we wish to carry out $L_p$ chaining with bracketing for, say, $p \in [1,2]$. Clearly $\pnorm{f}{L_2(P)}\leq \pnorm{f}{L_p(P)}^{p/2}$ and $\pnorm{f}{L_1(P)}\leq \pnorm{f}{L_p(P)}$. This naturally hints the following conjecture: for $\sigma>0$ not `too small',
\begin{align*}
\E \sup_{f \in \mathcal{F}(\sigma)}\abs{\G_n (f)} &\lesssim \inf_{0\leq \gamma\leq \sigma/2} \bigg\{\sqrt{n}\gamma+ \int_{\gamma^{p/2}}^{\sigma^{p/2}} \sqrt{\log \mathcal{N}_{[\,]}(\epsilon^{2/p},\mathcal{F},L_p(P))}\,\d{\epsilon}\bigg\}\\
&\asymp \inf_{0\leq \gamma\leq \sigma/2} \bigg\{\sqrt{n}\gamma+ \int_{\gamma}^{\sigma} u^{p/2-1} \sqrt{\log \mathcal{N}_{[\,]}(u,\mathcal{F},L_p(P))}\,\d{u}\bigg\}\\
&\lesssim \inf_{0\leq \gamma\leq \sigma/2}\bigg\{\sqrt{n}\gamma+ \int_{\gamma }^{\sigma } u^{(p-\alpha)/2-1}\d{u}\bigg\}\\
&\sim \inf_{0\leq \gamma\leq \sigma/2}\bigg\{\sqrt{n}\gamma+\biggabs{u^{(p-\alpha)/2}\Big\lvert_{u=\gamma}^{u=\sigma} }\bigg\}\\
&\sim 
\begin{cases}
\sigma^{\frac{p-\alpha}{2}}, &\alpha<p\\
\inf_{0\leq \gamma\leq \sigma/2}\big\{\sqrt{n}\gamma+\gamma^{-(\alpha-p)/2} \big\}, &\alpha>p
\end{cases}\\
&\sim
\begin{cases}
\sigma^{\frac{p-\alpha}{2}}, &\alpha<p\\
n^{\frac{\alpha-p}{2(\alpha+2-p)}}, &\alpha>p
\end{cases}
.
\end{align*}
Below we implement this heuristic program rigorously and identify the regime of $\sigma>0$ in which the above bound holds.

\begin{proof}[Proof of Theorem \ref{thm:upper_bound_ep}-(1)]
	The proof is inspired by the proof of \cite[Lemma 2.14.3]{van1996weak} that is originated in \cite{ossiander1987central}. We use the same notation for convenience of the readers. Without loss of generality we assume $\sigma = 2^{-q_0}$ for some $q_0 \in \N$. By the assumption, there exist nested partitions $\{\mathcal{F} = \cup_{i=1}^{N_q} \mathcal{F}_{q,i}\}_{q=q_0}^\infty$ such that for all $q\geq q_0$, (i) $\max_{1\leq i\leq N_q}\pnorm{ \sup_{f,g \in \mathcal{F}_{q,i}} \abs{f-g} }{L_p}\leq 2^{-q}$, and (ii) $\log N_{q} \lesssim_{L,\alpha} 2^{q\alpha}$. 	[Such nested partitions can be constructed as follows. First taking un-nested partitions $\{\mathcal{F} = \cup_{i=1}^{\bar{N}_q} \bar{\mathcal{F}}_{q,i}\}_{q=q_0}^\infty$ with $\bar{N}_q \leq \mathcal{N}_{[\,]}(2^{-q},\mathcal{F},L_p)\leq \exp(L\cdot 2^{q\alpha} )$. Then at level $q$,  we use the partition consisting of all intersections of form $\{\cap_{r=q_0}^q \bar{\mathcal{F}}_{r,i_r}: 1\leq i_r\leq \bar{N}_r, q_0\leq r\leq q\}$. The first property above is obvious. The second one follows as $N_q\leq \prod_{r=q_0}^q \bar{N}_r$.] 
	
	Pick any $f_{q,i} \in \mathcal{F}_{q,i}$. For any $f \in \mathcal{F}$, let $\pi_q f \equiv f_{q,i}$ and $\Delta_q f\equiv \sup_{f,g \in \mathcal{F}_{q,i}} \abs{f-g}$ if $f \in \mathcal{F}_{q,i}$. Let $a_q \equiv  2^{-q\{(p\wedge 2)+\alpha\}/2}$.  Now define the indicator functions $
	A_{q-1} f \equiv \bm{1}_{\Delta_{q_0} f\leq \sqrt{n} a_{q_0},\ldots, \Delta_{q-1} f\leq \sqrt{n} a_{q-1} }$, $B_{q} f\equiv \bm{1}_{\Delta_{q_0} f\leq \sqrt{n} a_{q_0},\ldots, \Delta_{q-1} f\leq \sqrt{n} a_{q-1}, \Delta_q f>\sqrt{n}a_q}$, and $
	B_{q_0} f \equiv \bm{1}_{\Delta_{q_0}f>\sqrt{n}a_{q_0}}$. In words, $A_{q-1}f$ indicates the region for which `Gaussian estimates' are valid up to level $q-1$ for $f$, while $B_q$ indicates the region for which the `Gaussian estimate' is first violated at level $q$ for $f$.

	By the last display in \cite[pp. 241]{van1996weak}, the following chaining holds for any $q_1>q_0$:
	\begin{align}\label{ineq:local_maximal_ineq_bracketing_1}
	f-\pi_{q_0} f &= (f-\pi_{q_0}f) B_{q_0} f + \sum_{q=q_0+1}^{q_1} (f-\pi_q f) B_q f \nonumber\\
	&\quad \quad + \sum_{q=q_0+1}^{q_1} (\pi_q f-\pi_{q-1}f) A_{q-1}f + (f-\pi_{q_1}f) A_{q_1}f. 
	\end{align}
	We bound the expectation of the above four terms when applied with the empirical process, and name them $(I)$-$(IV)$. Roughly speaking, $(III)$ is the term with `Gaussian behavior' due to the construction of $A_{q-1}f$, while for the terms $(I)$-$(II)$, the Gaussian estimate fails on $B_q f$ at level $q$. This will be compensated by the observation that the Gaussian estimate still holds at level $q-1$. The single term $(IV)$ terminates the chaining and will be handled via an $L_p$ control. Below we implement this rough idea precisely. 
	
	For $(I)$ and $(II)$, note that
	\begin{align}\label{ineq:local_maximal_ineq_bracketing_2}
	&\E \sup_{f \in \mathcal{F}} \abs{\G_n(f-\pi_{q}f)B_{q}f} \nonumber\\
	&\leq \E \sup_{f \in \mathcal{F}}  \sqrt{n}\Prob_n \abs{f-\pi_{q} f} B_{q} f+ \sqrt{n}\sup_{f \in \mathcal{F}} P\abs{f-\pi_{q} f}B_{q} f \nonumber\\
	&\leq \E \sup_{f \in \mathcal{F}}  \sqrt{n}\Prob_n \Delta_{q} f B_{q} f + \sqrt{n} \sup_{f \in \mathcal{F}} P \Delta_{q} f B_{q} f \nonumber\\
	&\leq \E \sup_{f \in \mathcal{F}}  \abs{\G_n (\Delta_{q} f B_{q} f) } +2 \sqrt{n} \sup_{f \in \mathcal{F}} P \Delta_{q} f B_{q} f.
	\end{align}
	Hence
	\begin{align}\label{ineq:local_maximal_ineq_bracketing_3}
	(I)+(II) &=\sum_{q=q_0}^{q_1} \E \sup_{f \in \mathcal{F}} \abs{\G_n(f-\pi_{q}f)B_{q}f}\nonumber \\
	&\leq \sum_{q=q_0}^{q_1} \E \sup_{f \in \mathcal{F}} \abs{\G_n (\Delta_q f B_q f)} + 2\sqrt{n} \sum_{q=q_0}^{q_1} \sup_{f \in \mathcal{F}} P \Delta_q f B_q f.
	\end{align}
	For $q=q_0$, we use the trivial bound $\Delta_{q_0} f B_{q_0} f\leq 1$. For $q\geq q_0+1$, we will however implement control at level $q-1$: As the partitions are nested, $\Delta_q f B_q f\leq \Delta_{q-1} f B_q f \leq \sqrt{n} a_{q-1}$ for all $q\geq q_0+1$. In summary, for $q\geq q_0$,
	\begin{align*}
	\Delta_{q_0} f B_{q_0} f \leq 1\wedge \sqrt{n} a_{q-1}.
	\end{align*}
	On the other hand,
	\begin{align*}
	\sup_{f \in \mathcal{F}} P(\Delta_{q}f B_{q}f)^2\leq \sup_{f \in \mathcal{F}} \big( P(\Delta_{q}f B_{q}f)^p\big)^{1\wedge (2/p)} \leq 2^{-q (p\wedge 2)}.
	\end{align*}
	By Bernstein inequality, for any $q\geq q_0$,
	\begin{align}\label{ineq:local_maximal_ineq_bracketing_4}
	&\E \sup_{f \in \mathcal{F}} \abs{\G_n (\Delta_q f B_q f)} \nonumber \\
	&\lesssim 2^{-q(p\wedge 2)/2} \sqrt{\log N_q} +n^{-1/2} (1\wedge \sqrt{n} a_{q-1}) \log N_q \nonumber\\
	& \lesssim 2^{-q \{ (p\wedge 2)-\alpha\}/2 } + (n^{-1/2}\wedge 2^{-q\{(p\wedge 2)+\alpha\}/2 })\cdot  2^{q\alpha} \asymp 2^{-q \{ (p\wedge 2)-\alpha\}/2 }.
	\end{align}
	Furthermore,
	\begin{align}\label{ineq:local_maximal_ineq_bracketing_5}
	P \Delta_{q} f B_{q} f &\leq P\Delta_{q} f \bm{1}_{ \Delta_{q} f>\sqrt{n} a_{q} }\leq  (\sqrt{n} a_{q})^{-1} P(\Delta_{q}f)^2 \nonumber \\
	&\leq n^{-1/2} a_{q}^{-1} 2^{-q(p\wedge 2)} = n^{-1/2} 2^{-q\{(p\wedge 2)-\alpha\}/2}. 
	\end{align}
	Combining (\ref{ineq:local_maximal_ineq_bracketing_3})-(\ref{ineq:local_maximal_ineq_bracketing_5}), we have
	\begin{align}\label{ineq:local_maximal_ineq_bracketing_6}
	(I)+(II)&\lesssim \sum_{q=q_0}^{q_1}  2^{-q\{(p\wedge 2)-\alpha\}/2} \nonumber\\
	&\asymp_{p,\alpha} 2^{-q_0 \{(p\wedge 2)-\alpha\}/2}\bm{1}_{\alpha <p\wedge 2}+ 2^{q_1\{\alpha-(p\wedge 2)\}/2} \bm{1}_{\alpha>p\wedge 2}. 
	\end{align}
	For $(III)$, we have Gaussian estimates as follows. Note that 
	\begin{align*}
	\abs{\pi_q f-\pi_{q-1}f} A_{q-1} f &\leq 1\wedge \Delta_{q-1} f A_{q-1} f\leq 1\wedge \sqrt{n} a_{q-1},\\
	P \big(\abs{\pi_q f-\pi_{q-1}f} A_{q-1} f\big)^2 &\leq P(\Delta_{q-1} f)^2 \leq 2^{-(q-1)(p\wedge 2)}.
	\end{align*}
	As the cardinality of $\{(\pi_q f-\pi_{q-1}f) A_{q-1} f: f \in \mathcal{F}\}$ is at most $N_qN_{q-1}\leq N_q^2$, by Bernstein inequality and a similar argument to (\ref{ineq:local_maximal_ineq_bracketing_4}), 
	\begin{align}\label{ineq:local_maximal_ineq_bracketing_7}
	(III) &=\sum_{q=q_0+1}^{q_1} \E \sup_{f \in \mathcal{F}} \abs{\G_n\big((\pi_q f-\pi_{q-1}f)A_{q-1}f\big)} \nonumber \\
	&\lesssim \sum_{q=q_0+1}^{q_1} 2^{-q \{ (p\wedge 2)-\alpha\}/2 } \nonumber\\
	&\asymp_{p,\alpha} 2^{-q_0 \{(p\wedge 2)-\alpha\}/2}\bm{1}_{\alpha <p/2}+ 2^{q_1\{\alpha-(p\wedge 2)\}/2} \bm{1}_{\alpha>p\wedge 2}. 
	\end{align}
	For $(IV)$, using similar arguments as in (\ref{ineq:local_maximal_ineq_bracketing_2}),
	\begin{align}\label{ineq:local_maximal_ineq_bracketing_8}
	(IV)& = \E \sup_{f \in \mathcal{F}} \abs{\G_n(f-\pi_{q_1}f)A_{q_1}f} \nonumber \\
	&\leq \E \sup_{f \in \mathcal{F}}  \abs{\G_n (\Delta_{q_1} f A_{q_1} f) } +2 \sqrt{n} \sup_{f \in \mathcal{F}} P \Delta_{q_1} f A_{q_1} f \nonumber \\
	&\lesssim 2^{-q_1 \{(p\wedge 2)-\alpha\}/2} + \sqrt{n} 2^{-q_1}. 
	\end{align}
	Here in the last inequality we used $P \Delta_{q_1} f A_{q_1} f \leq \pnorm{\Delta_{q_1} f A_{q_1} f }{L_p}\leq 2^{-q_1}$. Finally, using Bernstein's inequality again,
	\begin{align}\label{ineq:local_maximal_ineq_bracketing_9}
	\E \sup_{f \in \mathcal{F}} \abs{\G_n (\pi_{q_0} f)} &\lesssim 2^{-q_0 \{ (p\wedge 2)-\alpha\}/2}+ n^{-1/2} 2^{q_0\alpha}.
	\end{align}
	Combining  (\ref{ineq:local_maximal_ineq_bracketing_1}), (\ref{ineq:local_maximal_ineq_bracketing_6})-(\ref{ineq:local_maximal_ineq_bracketing_9}), we obtain
	\begin{align*}
	\E \sup_{f \in \mathcal{F}} \abs{\G_n(f)}
	& \lesssim 
	\begin{cases}
	2^{-q_0 \{ (p\wedge 2)-\alpha\}/2}+ n^{-1/2} 2^{q_0\alpha} + \sqrt{n} 2^{-q_1}, & \alpha <p\wedge 2,\\
	2^{q_1 \{ \alpha-(p\wedge 2)\}/2}+ n^{-1/2} 2^{q_0\alpha}  + \sqrt{n} 2^{-q_1}, & \alpha>p\wedge 2.
	\end{cases}
	\end{align*}
	The $\lesssim$ does not depend on $q_1$. Now optimize over $q_1>q_0$ to conclude. 
\end{proof}

\begin{remark}
\cite{pollard2002maximal} suggested the following key `recursive equality' (cf. \cite[pp. 1043, last display]{pollard2002maximal})
\begin{align}\label{ineq:pollard_recursive}
R_i T_i = R_{i+1}T_{i+1}-R_{i+1}T_i^cT_{i+1}+(R_i-R_{i+1})T_iT_{i+1}+R_iT_iT_{i+1}^c
\end{align}
for the chaining method of \cite{ossiander1987central}. See \cite{pollard2002maximal} for definitions of the above notation. Interestingly, (\ref{ineq:pollard_recursive}) can also be viewed as a one-step chaining of (\ref{ineq:local_maximal_ineq_bracketing_1}) from $i$ to $i+1$, as we may regard $f-\pi_i f = R_i$, $A_i f = T_i$, $A_{i+1} f = T_i T_{i+1}$ and $B_{i+1} f= T_i T_{i+1}^c$ by translating the notation used here to those in \cite{pollard2002maximal}. Some elementary algebra then reduces (\ref{ineq:local_maximal_ineq_bracketing_1}) to (\ref{ineq:pollard_recursive}).
\end{remark}

\subsection{Proof of Theorem \ref{thm:upper_bound_ep}-(3) }

The proof of (\ref{ineq:lower_bound_ep}) is divided into two steps:
\begin{enumerate}
	\item First, we establish a lower bound for the \emph{Gaussianized} empirical process; see Proposition \ref{prop:lower_bound_ep_set}. This can be done roughly via Sudakov minimization in the Gaussian regime, i.e., when $\sigma$ is not too small.
	\item Second, we will control the Gaussianized empirical process by the standard symmetrized empirical process from above. This step requires several subtle estimates as a naive bound would incur additional undesirable logarithmic factors. This is done via the help of a multiplier inequality derived in \cite{han2017sharp} (cf. Appendix \ref{section:tools}). Roughly speaking, the logarithmic factors can be removed for the Gaussianized empirical process at sample size $n$ as long as the empirical process has size no smaller than Gaussian maxima along the `entire path' from $1$ to $n$. Details see the proofs of Propositions \ref{prop:lower_bound_ep_set_donsker} and \ref{prop:lower_bound_ep_set_nolog}.
\end{enumerate}

We first prove the lower bound for Gaussianized empirical process.
\begin{proposition}\label{prop:lower_bound_ep_set}
	Let $\mathcal{F}\subset L_\infty(1)$. For any $\sigma \geq 50n^{-1/2}$ such that 
	\begin{align*}
	\log \mathcal{N}(\sigma/4,\mathcal{F}(\sigma),L_2(P))\leq n\sigma^2/4000,
	\end{align*} 
	we have
	\begin{align*}
	\E \sup_{f \in \mathcal{\mathcal{F}}(\sigma)} \biggabs{\frac{1}{\sqrt{n}}\sum_{i=1}^n g_i f(X_i) }\gtrsim \sigma \sqrt{\log \mathcal{N}(\sigma/2,\mathcal{F}(\sigma), L_2(P))}.
	\end{align*}
	Here $g_1,\ldots,g_n$ are i.i.d. $\mathcal{N}(0,1)$.
\end{proposition}

\begin{proof}[Proof of Proposition \ref{prop:lower_bound_ep_set}]
	By Sudakov minorization (cf. Lemma \ref{lem:sudakov_minor}), for any $\sigma>0$,
	\begin{align}\label{ineq:lower_bound_sudakov_1}
	\E \sup_{f \in \mathcal{F}(\sigma)} \biggabs{\frac{1}{\sqrt{n}}\sum_{i=1}^n g_i f(X_i) }\gtrsim \E \sigma \sqrt{\log \mathcal{N}(\sigma/10,\mathcal{F}(\sigma), L_2(\Prob_n))}.
	\end{align}
	We claim that for any $\sigma>0$ such that $\log \mathcal{N}(\sigma/4,\mathcal{F}(\sigma),L_2(P))\leq n\sigma^2/4000$,
	\begin{align}\label{ineq:lower_bound_sudakov_2}
	\Prob\left(\mathcal{N}(\sigma/10,\mathcal{F}(\sigma), L_2(\Prob_n))\geq \mathcal{N}(\sigma/2,\mathcal{F}(\sigma), L_2(P))\right)\geq 1-e^{-n\sigma^2/2000}.
	\end{align}
	To see this, let $f_1,\ldots,f_N$ be a maximal $\sigma/2$-packing set of $\mathcal{F}(\sigma)$ in the $L_2(P)$ metric, i.e. for $i\neq j$, $P(f_i-f_j)^2\geq \sigma^2/4$. Since $P(f_i-f_j)^4\leq 4P(f_i-f_j)^2\leq 16\sigma^2$, we apply Bernstein's inequality followed by a union bound to see that with probability at least $1-N^2 \exp(-t)$,
	\begin{align*}
	\max_{1\leq i\neq j\leq N} \bigg(nP(f_i-f_j)^2-\sum_{k=1}^n(f_i-f_j)^2(X_k) \bigg)\leq \frac{2t}{3}+\sqrt{32tn\sigma^2}.
	\end{align*}
	With $t=cn\sigma^2$ for a constant $c>0$ to be specified below, we obtain
	\begin{align*}
	&\Prob\bigg(\min_{1\leq i\neq j\leq N} \frac{1}{n}\sum_{k=1}^n(f_i-f_j)^2(X_k) \geq \sigma^2\left(1/4-2c/3-\sqrt{32c}\right) \bigg)\\
	&\geq 1-e^{2\log \mathcal{D}(\sigma/2,\mathcal{F}(\sigma),L_2(P))-cn\sigma^2}\geq 1-e^{2\log \mathcal{N}(\sigma/4,\mathcal{F}(\sigma),L_2(P))-cn\sigma^2},
	\end{align*}
	where $\mathcal{D}(\cdot,\cdot,\cdot)$ stands for the packing number. By choosing $c=1/10^3$ and $\log \mathcal{N}(\sigma/4,\mathcal{F}(\sigma),L_2(P))\leq n\sigma^2/4000$, we have
	\begin{align*}
	\Prob\bigg(\min_{1\leq i\neq j\leq N} \frac{1}{n}\sum_{k=1}^n(f_i-f_j)^2(X_k) \geq 0.04\sigma^2\bigg)\geq 1-\exp(-n\sigma^2/2000).
	\end{align*}
	This entails that $
	\mathcal{D}(\sigma/5, \mathcal{F}(\sigma),L_2(\Prob_n))\geq N\equiv \mathcal{D}(\sigma/2,\mathcal{F}(\sigma),L_2(P))$ with the above probability. 
	Hence for any $\sigma>0$ such that $\log \mathcal{N}(\sigma/4,\mathcal{F}(\sigma),L_2(P))\leq n\sigma^2/4000$, with probability at least $1-e^{-n\sigma^2/2000}$,
	\begin{align*}
	\mathcal{N}(\sigma/10, \mathcal{F}(\sigma),L_2(\Prob_n))&\geq \mathcal{D}(\sigma/5, \mathcal{F}(\sigma),L_2(\Prob_n))\\
	&\geq \mathcal{D}(\sigma/2,\mathcal{F}(\sigma),L_2(P))\geq \mathcal{N}(\sigma/2,\mathcal{F}(\sigma),L_2(P)),
	\end{align*}
	completing the proof of (\ref{ineq:lower_bound_sudakov_2}). Hence for any $\sigma \geq 50n^{-1/2}$ such that the entropy $\log \mathcal{N}(\sigma/4,\mathcal{F}(\sigma),L_2(P))\leq n\sigma^2/4000$, the claim of the proposition follows from (\ref{ineq:lower_bound_sudakov_1}) and (\ref{ineq:lower_bound_sudakov_2}).
\end{proof}

Next we eliminate the effect of the Gaussian multiplier. We need a technical lemma.

\begin{lemma}\label{lem:Gaussian_order_stat}
	Let $g_1,\ldots,g_n$ be i.i.d. $\mathcal{N}(0,1)$, and $\abs{g_{(n)}}\leq \ldots\leq \abs{g_{(1)}}$ be reversed order statistics of $\{\abs{g_1},\ldots,\abs{g_n}\}$. Then there exists an absolute constant $K>0$ such that for any $c \in (0,K^{-1})$, and $0\leq t\leq K^{-1} \sqrt{\log(1/c)}$, we have $
	\Prob\big(\abs{g_{(\floor{cn})}}\leq t \big)\leq e^{-c^2 n /K}$. 
\end{lemma}
\begin{proof}
	For notational convenience, we assume that $c n \in \N$. Let $\phi(t) \equiv \Prob(\abs{g_1}>t)$. By \cite[(2.23)]{gine2015mathematical}, $\sqrt{2/\pi}\cdot \frac{t}{t^2+1} e^{-t^2/2}\leq \phi(t)\leq \min\{1, \sqrt{2/\pi}\cdot t^{-1}\}e^{-t^2/2}$. Let $t_c>0$ be such that $\phi(t_c)=2c$. Then $t_c\leq \sqrt{2\log(1/2c)}$. By Bernstein's inequality, for $0\leq t\leq t_c$, $2c \leq \phi(t)\leq 1$, so
	\begin{align*}
	\Prob\big(\abs{g_{(cn)}}\leq t\big) &\leq \Prob\bigg(\sum_{i=1}^n \bm{1}_{\abs{g_i}>t}\leq cn\bigg) \\
	&= \Prob\bigg(\sum_{i=1}^n \big(\bm{1}_{\abs{g_i}>t}-\phi(t)\big)\leq-(\phi(t)-c)n\bigg)\\
	&\leq \exp\bigg(-\frac{(\phi(t)-c)^2n^2 }{2n\phi(t)+4 (\phi(t)-c)n/3}\bigg)\leq e^{- c^2 n/K},
	\end{align*}
	proving the claim. 
\end{proof}

\begin{proposition}\label{prop:lower_bound_ep_set_donsker}
	Let the conditions in Theorem \ref{thm:upper_bound_ep}-(3) hold for some $\alpha\in (0,1)$ and $L>0$ large enough. Then for $\sigma_n^2\geq c n^{-1/(\alpha+1)}$ with some constant $c>0$, $
	\E \sup_{f \in \mathcal{F}(\sigma_n)}\abs{\G_n(f)}\gtrsim_{\alpha} \sigma_n^{1-\alpha}$.
\end{proposition}

\begin{proof}
	By Proposition \ref{prop:lower_bound_ep_set}, the Gaussianized empirical process satisfies 
	\begin{align*}
	\E \sup_{f \in \mathcal{F}(\sigma_n)} \biggabs{\frac{1}{\sqrt{n}}\sum_{i=1}^n g_i f(X_i) }\gtrsim  \sigma_n \sqrt{\log \mathcal{N}(\sigma_n/2,\mathcal{F}(\sigma_n), L_2(P))}\geq C_1^{-1} \sigma_n^{1-\alpha}.
	\end{align*}
	Suppose that $\sigma_n^2\leq c$. Without loss of generality, we assume that $\sigma_n^2\equiv \sigma_n(\gamma)^2=c n^{-\gamma}$ for some $0\leq \gamma\leq 1/(\alpha+1)$ and define $\sigma_k^2 \equiv c k^{-\gamma}$. We first prove the following claim: there exists some $c_1\equiv c_1(c,\alpha)>0$ such that for any $0\leq \gamma\leq 1/(\alpha+1)$, 
	\begin{align}\label{ineq:lower_bound_ep_set_Donsker_1}
	\E \sup_{f \in \mathcal{F}(\sigma_n)} \biggabs{\frac{1}{\sqrt{n}}\sum_{i=1}^n \epsilon_i f(X_i) }\geq c_1 \sigma_n^{1-\alpha}. 
	\end{align}
	To this end, let $
	a_n\equiv (\sigma_n^{1-\alpha})^{-1} \E \sup_{f \in \mathcal{F}(\sigma_n)} \bigabs{\frac{1}{\sqrt{n}}\sum_{i=1}^n \epsilon_i f(X_i) }$. 
	Then by the local maximal inequality (cf. Theorem \ref{thm:upper_bound_ep}-(1)), we see that $\sup_{k \in \N} a_k \leq C_2\equiv C_2(\alpha)$. Since $\sqrt{n}\sigma_n^{1-\alpha}=c^{(1-\alpha)/2} n^\beta$ where $\beta\equiv \beta(\alpha,\gamma)\equiv \frac{1}{2}\big(1-(1-\alpha)\gamma\big) \in [\alpha/(1+\alpha),1/2]$, we have by Lemma \ref{lem:interpolation_ineq} that for a constant $c'>0$ to be determined later,
	\begin{align}\label{ineq:lower_bound_ep_set_Donsker_3}
	C_1^{-1}  n^\beta &\leq c^{-(1-\alpha)/2}\E \sup_{f \in \mathcal{F}(\sigma_n)} \biggabs{\sum_{i=1}^n g_i f(X_i) } \nonumber\\
	&\leq c^{-(1-\alpha)/2}\E \left[ \sum_{k=1}^n  (\abs{g_{(k)}}-\abs{g_{(k+1)}}) \E \sup_{f \in \mathcal{F}(\sigma_k)} \biggabs{\sum_{i=1}^k \epsilon_i f(X_i) } \right]\nonumber\\
	&\leq \E \bigg[ \sum_{k=1}^{ \floor{c'n}-1 }  (\abs{g_{(k)}}-\abs{g_{(k+1)}}) a_k k^\beta \bigg]\nonumber\\
	&\qquad+ \E \bigg[ \sum_{k={\floor{c'n} }}^n  (\abs{g_{(k)}}-\abs{g_{(k+1)}}) a_k k^\beta \bigg]\equiv (I)+(II).
	\end{align}	
	For $(I)$ in (\ref{ineq:lower_bound_ep_set_Donsker_3}), using the same notation as in the proof of Lemma \ref{lem:Gaussian_order_stat},
	\begin{align*}
	(I) &\leq C_2 \cdot \E \bigg[ \sum_{k=1}^{\floor{c'n}-1} (\abs{g_{(k)}}-\abs{g_{(k+1)}}) k^\beta \bigg]\\
	&\leq  C_2\cdot   \E \bigg[\sum_{k=1}^{\floor{c'n}-1} \int_{\abs{g_{(k+1)}} }^{\abs{g_{(k)}}} k^\beta\ \d{t}\bigg] \\
	& \leq  C_2\cdot  \E \int_0^\infty \bigg(\sum_{i=1}^n \bm{1}_{\abs{g_i}\geq t}\bm{1}_{\abs{g_{(\floor{c'n})}}\leq t\leq \abs{g_{(1)}}}\bigg)^\beta \ \d{t} \nonumber\\
	& \leq C_2 n^\beta\cdot  \int_0^\infty \bigg(\Prob(\abs{g_1}>t) \Prob\big(\abs{g_{(\floor{c'n})}}\leq t\leq \abs{g_{(1)}}\big)\bigg)^{\beta/2}\ \d{t} \nonumber \\ 
	& \leq C_2 n^\beta \bigg( \int_0^{K^{-1} \sqrt{\log(1/c')}} \phi(t)^{\beta/2} \Prob\big(\abs{g_{(\floor{c'n})}}\leq t\big)^{\beta/2}\ \d{t} +  \int_{K^{-1}\sqrt{\log(1/c')}}^{\infty} \phi(t)^{\beta/2}\ \d{t}\bigg) \nonumber\\
	&\leq C_2 n^\beta \bigg(\int_0^{K^{-1} \sqrt{\log(1/c')}}  e^{-\beta t^2/4-\beta (c')^2 n/2K}\ \d{t} + \int_{K^{-1}\sqrt{\log(1/c')}}^{\infty} e^{-\beta t^2/4}\ \d{t}\bigg)\nonumber\\
	&\leq C_3 n^\beta \big( e^{-(c')^2 n/C_3}+ e^{-\log(1/c')/C_3}\big)\leq (C_1^{-1}/2) n^\beta, \nonumber
	\end{align*}
	by choosing $c'\equiv \exp\big(-C_3 \log(4 C_1C_3)\big)$ and $n\geq C_3 \log(4C_1C_3)/(c')^2$. On the other hand, for $(II)$ in (\ref{ineq:lower_bound_ep_set_Donsker_3}), we have
	\begin{align*}
	(II)&\leq \big(\max_{\floor{c'n}\leq k\leq n} a_k \big) \E \bigg[ \sum_{k=1}^n  (\abs{g_{(k)}}-\abs{g_{(k+1)}}) k^\beta \bigg]
	\leq \big(\max_{\floor{c'n}\leq k\leq n} a_k \big) G_\alpha n^\beta,
	\end{align*}
	where $G_\alpha = \int_0^\infty \big(\Prob(\abs{g_1}>t)\big)^{\alpha/(1+\alpha)}\ \d{t}<\infty$ since Gaussian random variables have finite moments of any order, and the last inequality follows from Jensen's inequality. Combining the above displays we see that
	\begin{align*}
	1/(2C_1 G_a)&\leq \max_{\floor{c' n}\leq k\leq n} a_k \leq \max_{\floor{c' n}\leq k\leq n}  (\sigma_k^{1-\alpha})^{-1} \E \sup_{f \in \mathcal{F}(\sigma_k)} \biggabs{\frac{1}{\sqrt{k}}\sum_{i=1}^k \epsilon_i f(X_i) }\\
	&\leq \sigma_{n}^{-(1-\alpha)}\cdot \sqrt{1/(c'
		-1/n)}\cdot  \E \sup_{f \in \mathcal{F}(\sigma_{\floor{c' n}})}\biggabs{\frac{1}{\sqrt{n}}\sum_{i=1}^n \epsilon_i f(X_i) },
	\end{align*}
	where in the last inequality we used Jensen's inequality. This proves our claim (\ref{ineq:lower_bound_ep_set_Donsker_1}) by adjusting the constant and choosing $n\geq 2/c'$. Now by de-symmetrization inequality (cf. \cite[Lemma 2.3.6]{van1996weak}), we have that
	\begin{align}\label{ineq:lower_bound_ep_set_Donsker_2}
	\E \sup_{f \in \mathcal{F}(\sigma_n)} \biggabs{\frac{1}{\sqrt{n}}\sum_{i=1}^n \epsilon_i f(X_i) }\leq 2\E \sup_{f \in \mathcal{F}(\sigma_n)} \abs{\G_n(f)}+2\sigma_n.
	\end{align}
	For $\sigma_n\leq (c_1/4)^{1/\alpha}\wedge c^{1/2}$, the claim of the proposition follows from (\ref{ineq:lower_bound_ep_set_Donsker_1}) and (\ref{ineq:lower_bound_ep_set_Donsker_2}). On the other hand, the claim is trivial for $\sigma_n> (c_1/4)^{1/\alpha}\wedge c^{1/2}$.
\end{proof}

\begin{remark}
	From the proof of Proposition \ref{prop:lower_bound_ep_set_donsker}, the bracketing entropy upper bound is only used to prove $\sup_{k \in \N} a_k \leq C_2\equiv C_2(\alpha)$. This means that we may impose instead a uniform entropy upper bound condition as in \cite{gine2006concentration} in the regime $0<\alpha<1$.
\end{remark}

\begin{proposition}\label{prop:lower_bound_ep_set_nolog}
	Let the conditions in Theorem \ref{thm:upper_bound_ep}-(3) hold for some $\alpha>1$ and $L>0$ large enough. Then for $\sigma_n^2\equiv c n^{-1/(\alpha+1)}$ with some constant $c>0$, $
	\E \sup_{f \in \mathcal{F}(\sigma_n)}\abs{\G_n(f)}\gtrsim n^{(\alpha-1)/2(\alpha+1)}.$
\end{proposition}
\begin{proof}
	Proposition \ref{prop:lower_bound_ep_set} shows that 
	\begin{align*}
	\E \sup_{f \in \mathcal{F}(\sigma_n)} \biggabs{\frac{1}{\sqrt{n}}\sum_{i=1}^n g_i f(X_i) }\gtrsim  \sigma_n \sqrt{\log \mathcal{N}(\sigma_n/2,\mathcal{F}(\sigma_n), L_2(P))}\gtrsim n^{(\alpha-1)/2(\alpha+1)}.
	\end{align*}
	Now applying Lemma \ref{lem:interpolation_ineq} in the following form,
	\begin{align*}
	\E \sup_{f \in \mathcal{F}(\sigma_n)} \biggabs{\frac{1}{\sqrt{n}}\sum_{i=1}^n g_i f(X_i) }&\lesssim \max_{1\leq k\leq n} \E \sup_{f \in \mathcal{F}(\sigma_n)} \biggabs{\frac{1}{\sqrt{k}}\sum_{i=1}^k \epsilon_i f(X_i) },
	\end{align*}
	we see that for some $K>0$,
	\begin{align*}
	\max_{1\leq k\leq n} \E \sup_{f \in \mathcal{F}(\sigma_n)} \abs{\G_k(f)}\geq K^{-1} n^{(\alpha-1)/2(\alpha+1)}.
	\end{align*}
	On the other hand, by enlarging $K$ if necessary, Theorem \ref{thm:upper_bound_ep}-(1) entails that $
	\E \sup_{f \in \mathcal{F}(\sigma_k)} \abs{\G_k(f)} \leq K\cdot k^{(\alpha-1)/2(\alpha+1)}$, 
	and hence
	\begin{align*}
	\max_{1\leq k\leq n} \E \sup_{f \in \mathcal{F}(\sigma_k)} \abs{\G_k(f)}\leq K n^{(\alpha-1)/2(\alpha+1)}
	\end{align*}
	by the assumption $\alpha> 1$. Combining the upper and lower estimates we see that 
	\begin{align*}
	K^{-1}n^{(\alpha-1)/2(\alpha+1)}&\leq \max_{1\leq k\leq n} \E \sup_{f \in \mathcal{F}(\sigma_n)} \abs{\G_k(f)}\\
	&\leq \max_{1\leq k\leq n}\E \sup_{f \in \mathcal{F}(\sigma_k)} \abs{\G_k(f)} \leq Kn^{(\alpha-1)/2(\alpha+1)}.
	\end{align*}
	Now  we will argue that the $\max$ operator can be `eliminated'. To this end, let $a_k\equiv \E \sup_{f \in \mathcal{F}(\sigma_n)} \abs{\G_k(f)}$ and $\beta \equiv (\alpha-1)/2(\alpha+1)$ for notational convenience. Let $k_n\equiv \argmax_{1\leq k\leq n} a_k$. We claim that $k_n \in [cn,n]$ where $c=K^{-2/\beta}\in (0,1)$. To see this, we only need to note $
	K^{-1} n^\beta \leq \max_{1\leq k\leq n} a_k=a_{k_n}\leq K k_n^\beta$, 
	which entails $k_n^\beta \geq K^{-2} n^\beta$. Hence
	\begin{align*}
	K^{-1} n^{(\alpha-1)/2(\alpha+1)}\leq \E \sup_{f \in \mathcal{F}(\sigma_n)} \abs{\G_{k_n}(f)}\leq \frac{1}{\sqrt{c}} \E \sup_{f \in \mathcal{F}(\sigma_n)} \abs{\G_n(f)}
	\end{align*}
	where the last inequality follows from Jensen's inequality, proving the claim.
\end{proof}

\begin{proof}[Proof of Theorem \ref{thm:upper_bound_ep}-(3)]
	The claims follow by combining Propositions \ref{prop:lower_bound_ep_set_donsker} and \ref{prop:lower_bound_ep_set_nolog}.
\end{proof}

\section{Proofs for Section \ref{section:application}}\label{section:proof_application}

\subsection{Proof of Theorem \ref{thm:rate_optimal_lse_reg}}

Before the proof of Theorem \ref{thm:rate_optimal_lse_reg}, we need following:
\begin{lemma}\label{lem:lse_additive_error}
	Consider the regression model (\ref{eqn:reg_model}) and the least squares estimator $\widehat{C}_n$ in (\ref{def:lse_reg_add}). Suppose that $(\xi_1,X_1),\ldots,(\xi_n,X_n)$ are i.i.d. random vectors with $\E[\xi_1|X_1]=0$ and $\E[\xi_1^2|X_1]\lesssim 1\vee \pnorm{\xi_1}{2}^2$ almost surely.  Further assume that
	\begin{align}\label{ineq:ep_additive_error}
	&\E \sup_{C \in \mathscr{C}: P\abs{C\Delta C_0}\leq \delta^2} \biggabs{\frac{1}{\sqrt{n}}\sum_{i=1}^n \epsilon_i(\bm{1}_{C}-\bm{1}_{C_0})(X_i) } \nonumber\\
	&\qquad \bigvee \E \sup_{C \in \mathscr{C}: P\abs{C\Delta C_0}\leq \delta^2} \biggabs{\frac{1}{\sqrt{n}}\sum_{i=1}^n \xi_i(\bm{1}_{C}-\bm{1}_{C_0})(X_i) }\lesssim \phi_n(\delta),
	\end{align}
	hold for some $\phi_n$ such that $\delta\mapsto \phi_n(\delta)/\delta$ is non-increasing. 
	Then  $\E_{C_0} P\abs{\widehat{C}_n\Delta C_0}=\mathcal{O}(\delta_n^2)$
	holds for any $\delta_n\geq n^{-1/2}\max\{1, \pnorm{\xi_1}{2}, \E^{1/8}\max_{1\leq i\leq n}\abs{\xi_i}^4\}$ such that $\phi_n(\delta_n)\leq \sqrt{n}\delta_n^2$, where the constant in $\mathcal{O}$ only depends on the constants in (\ref{ineq:ep_additive_error}).
\end{lemma}
\begin{proof}
See Appendix \ref{section:proof_lemma}.
\end{proof}

\begin{proof}[Proof of Theorem \ref{thm:rate_optimal_lse_reg}]
	By Lemma \ref{lem:lse_additive_error} the risk of the least squares estimator
	\begin{align*}
	\delta_n^2\equiv \sup_{C_0 \in \mathscr{C}} \E_{C_0} P\abs{\widehat{C}_n\Delta C_0}=\sup_{C_0 \in \mathscr{C}} \E_{C_0} \int (\bm{1}_{\widehat{C}_n}-\bm{1}_{C_0})^2\ \d{P}
	\end{align*}
	can be solved by estimating the empirical processes in (\ref{ineq:ep_additive_error}). Since the global entropy estimate is translation invariant (i.e. the metric entropy of $\{\bm{1}_C-\bm{1}_{C_0}: C \in \mathscr{C}\}$ is the same as that of $\{\bm{1}_C: C \in \mathscr{C}\}$), by Theorem \ref{thm:upper_bound_ep_set}, we obtain an estimate for the Rademacher randomized empirical process:
	\begin{align*}
	\sup_{C_0 \in \mathscr{C}}\E \sup_{C \in \mathscr{C}: P\abs{C\Delta C_0}\leq \delta_n^2} \biggabs{\frac{1}{\sqrt{n}}\sum_{i=1}^n \epsilon_i(\bm{1}_{C}-\bm{1}_{C_0})(X_i) }  \lesssim \max\{\delta_n^{1-\alpha}, n^{(\alpha-1)/2(\alpha+1)}\}.
	\end{align*}
	It is now easy to see that the choice $\delta_n^2\asymp n^{-1/(\alpha+1)}$ leads to an upper bound of the above display on the desired order $\sqrt{n}\delta_n^2$. Note that the bound continues to hold when the left hand side of the above display is replaced with expected supremum without localization over $C \in \mathscr{C}$ for $P\abs{C\Delta C_0}\leq \delta_n^2$. The Gaussian randomized empirical process can be handled via the multiplier inequality Lemma \ref{lem:multiplier_ineq} by letting $\psi_n(t)\equiv \psi(t)\equiv t^{\alpha/(\alpha+1)}$, whence
	\begin{align*}
	\sup_{C_0 \in \mathscr{C}}\E \sup_{C \in \mathscr{C}} \biggabs{\frac{1}{\sqrt{n}}\sum_{i=1}^n \xi_i(\bm{1}_{C}-\bm{1}_{C_0})(X_i) }  \lesssim  n^{(\alpha-1)/2(\alpha+1)},
	\end{align*}
    completing the proof.
\end{proof}

\subsection{Proof of Theorem \ref{thm:tsybakov_example}}

We need the following analogy of Lemma \ref{lem:lse_additive_error} before proving Theorem \ref{thm:tsybakov_example}.

\begin{lemma}\label{lem:lse_dependent_bounded_error}
	Consider the regression model (\ref{eqn:reg_multi_model}) and the least squares estimator $\widehat{C}_n$ in (\ref{def:lse_reg_multi}). Further assume that
	\begin{align*}
	&\E \sup_{C \in \mathscr{C}: P\abs{C\Delta C_0}\leq \delta^2} \biggabs{\frac{1}{\sqrt{n}}\sum_{i=1}^n \epsilon_i(\bm{1}_{C}-\bm{1}_{C_0})(X_i) }\\
	&\qquad \bigvee \E \sup_{C \in \mathscr{C}: P\abs{C\Delta C_0}\leq \delta^2} \biggabs{\frac{1}{\sqrt{n}}\sum_{i=1}^n \epsilon_i(\bm{1}_{C\cap C_0}-\bm{1}_{C_0})(X_i) } \lesssim \phi_n(\delta),
	\end{align*}
	holds for some $\phi_n$ such that $\delta\mapsto \phi_n(\delta)/\delta$ is non-increasing. 
	Then  $\E_{C_0} P \abs{\widehat{C}_n\Delta C_0}=\mathcal{O}(\delta_n^2)$
	holds for any $\delta_n\geq n^{-1/2}\max\{1, \pnorm{\xi_1}{2}, \E^{1/8}\max_{1\leq i\leq n}\abs{\xi_i}^4\}$ such that $\phi_n(\delta_n)\leq \sqrt{n}\delta_n^2$, where the constant in $\mathcal{O}$ only depends on the constants in the above inequality. 
\end{lemma}
\begin{proof}
	See Appendix \ref{section:proof_lemma}.
\end{proof}

\begin{proof}[Proof of Theorem \ref{thm:tsybakov_example}]
	The proof follows by Lemma \ref{lem:lse_dependent_bounded_error} and similar arguments as in the proof of Theorem \ref{thm:rate_optimal_lse_reg}.
\end{proof}

\subsection{Proof of Theorem \ref{thm:excess_risk_binary}}
We need some further notations to implement this program. For any $g \in \mathcal{G}$, write $f_g(x,y)\equiv \bm{1}_{y\neq g(x)}$. Let $\mathcal{G}(\delta)\equiv \{g \in \mathcal{G}: \mathcal{E}_P(g)\leq \delta\}$. Let $\ell$ be the smallest integer such that $r_n^2 2^\ell\geq 1$, and for any $1\leq j\leq \ell$, let $\mathcal{F}_j\equiv \{f_{g_1}-f_{g_2}: g_1,g_2 \in \mathcal{G}(r_n^2 2^j)\}$. 
\begin{lemma}\label{lem:excess_risk_estimate}
	Suppose $\mathcal{G}\equiv \{\bm{1}_C:C \in \mathscr{C}\}$ satisfies the same entropy condition as in Theorem \ref{thm:excess_risk_binary}. Then 
	\begin{align*}
	\Prob\bigg(\max_{1\leq j\leq \ell} \frac{\sup_{f \in \mathcal{F}_j} \abs{\Prob_n(f)-P(f)}  }{r_n^2 2^j}\geq c\bigg(\frac{1}{4}+K\sqrt{\frac{s}{nr_n^2} } +K\frac{s}{nr_n^2}\bigg) \bigg)\leq K' \exp(-s/K')
	\end{align*}
	holds for some constants $K, K'>0$ provided $r_n^2\cdot  n^{1/(\alpha+1)}\geq K''$ for a large enough constant $K''>0$ depending on $c>0$ in (\ref{ineq:margin_cond}) only. 
\end{lemma}
\begin{proof}[Proof of Lemma \ref{lem:excess_risk_estimate}]
	By Talagrand's concentration inequality (cf. Appendix \ref{section:tools}), with $\sigma_j^2 \equiv \sup_{f \in \mathcal{F}_j} \pnorm{f}{L_2(P)}^2$,
	\begin{align*}
	\Prob\bigg[\sup_{f \in \mathcal{F}_j} \abs{\G_n(f)}\geq K\bigg(\E \sup_{f \in \mathcal{F}_j} \abs{\G_n(f)}+\sqrt{\sigma_j^2 s_j}+\frac{s_j}{\sqrt{n}} \bigg) \bigg]\leq K\exp(-s_j/K).
	\end{align*}
	Let $\mathscr{S}\equiv \{S: f_g=\bm{1}_S, g \in \mathcal{G}\}$. Note that for $g_1=\bm{1}_{C_1},g_2=\bm{1}_{C_2} \in \mathcal{G}$,  where $C_1,C_2 \in \mathscr{C}$, we have $f_{g_1}=\bm{1}_{S_1},f_{g_2}=\bm{1}_{S_2}$, and hence
	\begin{align*}
	P(S_1\Delta S_2)=P(f_{g_1}-f_{g_2})^2\leq P(g_1-g_2)^2=P(C_1\Delta C_2).
	\end{align*}
	This shows that $\mathcal{N}_{I}(\epsilon, \mathscr{S},P)\leq \mathcal{N}_I(\epsilon, \mathscr{C},P)$. Furthermore, for any $g \in \mathcal{G}(r_n^22^j)$, let $S \in \mathscr{S}$ be such that $f_g=\bm{1}_S$. Then similar to the above display, we have
	\begin{align*}
	P(S\Delta S_0)\leq \pnorm{g-g_0}{L_2(P)}^2\leq c^{-1} r_n^22^j,
	\end{align*}
	where the last inequality follows from the margin condition. Now by Theorem \ref{thm:upper_bound_ep_set}, we obtain
	\begin{align*}
	\E \sup_{f \in \mathcal{F}_j} \abs{\G_n(f)}&\lesssim \E \sup_{g \in \mathcal{G}(r_n^22^j)} \abs{\G_n(f_g)}\leq \E \sup_{S \in \mathscr{S}: P(S\Delta S_0)\leq c^{-1} r_n^2 2^j} \abs{\G_n(S)}\\
	&\lesssim \max\{(r_n^2 2^j)^{(1-\alpha)/2}, n^{(\alpha-1)/2(\alpha+1)}\}.
	\end{align*}
	On the other hand, 
	\begin{align*}
	\sigma_j^2 &\equiv \sup_{f \in \mathcal{F}_j} \pnorm{f}{L_2(P)}^2=\sup_{g_1, g_2 \in \mathcal{G}(r_n^22^j)} \pnorm{f_{g_1}-f_{g_2}}{L_2(P)}^2\\
	&\leq 4\sup_{g \in \mathcal{G}(r_n^22^j)} \pnorm{g-g_0}{L_2(P)}^2\leq 4c^{-1} \sup_{g \in \mathcal{G}(r_n^22^j)}\mathcal{E}_P(g)\leq 4c^{-1} r_n^22^j.
	\end{align*}
	This implies that with $s_j=s2^j$,
	\begin{align*}
	&\Prob\bigg[ \frac{\sup_{f \in \mathcal{F}_j} \abs{\Prob_n(f)-P(f)}}{r_n^2 2^j} \geq K_c r_n^{-2}2^{-j}\bigg( \max\big\{n^{-1/2}(r_n^2 2^j)^{(1-\alpha)/2}, n^{-1/(\alpha+1)}\big\}\\
	&\qquad\qquad\qquad\qquad\qquad +n^{-1/2} (r_n^22^j)^{1/2} \sqrt{s}2^{j/2}+n^{-1} s 2^j\bigg)\bigg]\leq K\exp(-s2^j/K).
	\end{align*}
	Note that
	\begin{align*}
	&r_n^{-2}2^{-j}\left( \max\big\{n^{-1/2}(r_n^2 2^j)^{(1-\alpha)/2}, n^{-1/(\alpha+1)}\big\} +n^{-1/2} (r_n^2 2^j)^{1/2} \sqrt{s}2^{j/2}+n^{-1} s 2^j\right)\\
	&\leq \max\left\{\frac{1}{\sqrt{n}r_n^{\alpha+1}},\frac{1}{r_n^2 n^{1/(\alpha+1)}}\right\}+\sqrt{\frac{s}{nr_n^2} } +\frac{s}{nr_n^2} \leq \frac{c}{4K_c} +\sqrt{\frac{s}{nr_n^2} } +\frac{s}{nr_n^2} 
	\end{align*}
	under the assumption. Now a union bound leads to the desired claim.
\end{proof}

\begin{proof}[Proof of Theorem \ref{thm:excess_risk_binary}]
	Given the estimate in Lemma \ref{lem:excess_risk_estimate}, the proof of the theorem closely follows that of \cite[Theorem 7.1]{gine2006concentration}. We provide some details for the convenience of the reader. On the event 
	\begin{align*}
	E \equiv \left\{\max_{1\leq j\leq \ell} \frac{\sup_{f \in \mathcal{F}_j} \abs{\Prob_n(f)-P(f)}  }{r_n^2 2^j}\leq c\left(\frac{1}{4}+K\sqrt{\frac{s}{nr_n^2} } +K\frac{s}{nr_n^2}\right)  \right\},
	\end{align*}
	we have for any $g \in \mathcal{G}(r_n^22^j)\setminus \mathcal{G}(r_n^22^{j-1})$ and $g' \in \mathcal{G}(\sigma)$ for some $0<\sigma<r_n^2 2^j$,
	\begin{align*}
	\mathcal{E}_P(g) &= P (f_g-f_{g'})+\big[P(f_{g'})-Pf_{g_0}\big]\leq P (f_g-f_{g'})+\sigma\\
	&\leq \Prob_n (f_g-f_{g'})+\sigma+ \sup_{f \in \mathcal{F}_j}\abs{(\Prob_n-P)(f)}\\
	&\leq \mathcal{E}_{\Prob_n}(g)+\sigma+ c\left(\frac{1}{4}+K\sqrt{\frac{s}{nr_n^2} } +K\frac{s}{nr_n^2} \right) r_n^2 2^j\\
	&\leq \mathcal{E}_{\Prob_n}(g)+\sigma+ \left(\frac{1}{4}+K\sqrt{\frac{s}{nr_n^2} } +K\frac{s}{nr_n^2} \right)2\mathcal{E}_P(g).
	\end{align*}
	Since $\sigma>0$ is taken arbitrarily, we see that on the event $E$, it holds that
	\begin{align}\label{ineq:excess_risk_binary_1}
	\frac{\mathcal{E}_{\Prob_n}(g)}{\mathcal{E}_P(g)}\geq 1-\left(\frac{1}{2}+2K\sqrt{\frac{s}{nr_n^2} } +2K\frac{s}{nr_n^2}\right)
	\end{align}
	for all $g \in \mathcal{G}$ such that $\mathcal{E}_P(g)\geq r_n^2$. Furthermore, the above display entails that on the event $E$, we necessarily have $\mathcal{E}_P(\widehat{g}_n)<r_n^2$ for $n$ large enough. Hence for any $g \in \mathcal{G}(r_n^22^j)\setminus \mathcal{G}(r_n^2 2^{j-1})$, we have
	\begin{align*}
	\mathcal{E}_{\Prob_n}(g)&=\Prob_n(f_g)-\Prob_n(f_{\widehat{g}_n})\leq Pf_g-Pf_{\widehat{g}_n}+\sup_{f \in \mathcal{F}_j}\abs{(\Prob_n-P)(f)}\\
	&\leq \mathcal{E}_{P}(g)+ \left(\frac{1}{4}+K\sqrt{\frac{s}{nr_n^2} } +K\frac{s}{nr_n^2} \right)2\mathcal{E}_P(g).
	\end{align*}
	This entails that
	\begin{align}\label{ineq:excess_risk_binary_2}
	\frac{\mathcal{E}_{\Prob_n}(g)}{\mathcal{E}_P(g)}\leq  1+\left(\frac{1}{2}+2K\sqrt{\frac{s}{nr_n^2} } +2K\frac{s}{nr_n^2}\right).
	\end{align}
	The proof of the claim is complete by combining (\ref{ineq:excess_risk_binary_1})-(\ref{ineq:excess_risk_binary_2}) along with Lemma \ref{lem:excess_risk_estimate}.
\end{proof}

\subsection{Proof of Theorem \ref{thm:iso_reg}}

\begin{lemma}\label{lem:trunc_iso_reg}
	It holds for $C>0$ large enough that
	\begin{align*}
	\E_{f_0}\pnorm{\widehat{f}_n-f_0}{L_2(P)}^2&\leq \E_{f_0}\pnorm{\widehat{f}_n-f_0}{L_2(P)}^2\bm{1}_{\pnorm{\widehat{f}_n-f_0}{\infty}\leq C\sqrt{\log n}}+\mathcal{O}(n^{-1}).
	\end{align*}
	The $\mathcal{O}$ term is uniform in $f_0 \in \mathcal{M}_d\cap L_\infty(1)$.
\end{lemma}
\begin{proof}
	Fix $f_0 \in \mathcal{M}_d \cap L_\infty(1)$. By \cite[Lemma 10, supplement]{han2017isotonic}, $\sup_{x \in [0,1]^d} (\widehat{f}_n-f_0)(x) \leq \max_{1\leq i\leq n} Y_i +\pnorm{f_0}{\infty}\leq 2+ \max_{1\leq i\leq n} \xi_i$. Hence with $Z_n\equiv  \pnorm{\widehat{f}_n-f_0}{\infty}$, for $u$ large, $\Prob(Z_n>u\sqrt{\log n})\leq e^{-c u^2 \log n}$ for some $c>0$. In particular, $\E Z_n^4 \lesssim \log^2 n$. Now the claim of the lemma follows by noting that
	\begin{align*}
	&\E_{f_0}\pnorm{\widehat{f}_n-f_0}{L_2(P)}^2\bm{1}_{\pnorm{\widehat{f}_n-f_0}{\infty}> C\sqrt{\log n}}\leq \E Z_n^2 \bm{1}_{Z_n>C\sqrt{\log n}}\\
	& \leq \sqrt{\E Z_n^4}\cdot \sqrt{\Prob(Z_n>C\sqrt{\log n})} \leq \log n\cdot e^{-cC^2 \log n/2} = \mathcal{O}(n^{-1})
	\end{align*}
	for $C>0$ large.
\end{proof}

Recall that $B\subset \R^d$ is a lower (resp. upper) set if and only if for all $x\in B, y\in \R^d$ with $y_i\leq x_i$ (resp. $y_i\geq x_i$), $i=1,\ldots,d$, we have $y \in B$. Let $\mathcal{LL}_d$ be the collection of all upper and lower sets in $\R^d$ and $\mathcal{L}_d = \{B\cap [0,1]^d: B \in \mathcal{LL}_d\}$.

\begin{proof}[Proof of Theorem \ref{thm:iso_reg}]
	First consider $d\geq 3$. By Lemma \ref{lem:trunc_iso_reg}, we only need to compute an upper bound for $\E_{f_0}\pnorm{\widehat{f}_n-f_0}{L_2(P)}^2\bm{1}_{\pnorm{\widehat{f}_n-f_0}{\infty}\leq C\sqrt{\log n}}\leq \bar{r}_n^2$. Similar to the proof of Lemma \ref{lem:lse_additive_error}, this can be done by evaluating the size of two empirical processes
	\begin{align}\label{ineq:iso_reg_1}
	&\E \sup_{ \substack{f \in \mathcal{M}_d\cap L_\infty(C \sqrt{\log n}):\\ \pnorm{f-f_0}{L_2(P)}\leq \bar{r}_n}} \abs{\G_n(\xi (f-f_0))} \nonumber \\
	&\qquad \bigvee \E \sup_{ \substack{f \in \mathcal{M}_d\cap L_\infty(C \sqrt{\log n}):\\ \pnorm{f-f_0}{L_2(P)}\leq \bar{r}_n}} \abs{\G_n((f-f_0)^2)}\lesssim \sqrt{n} \bar{r}_n^2.
	\end{align}
	Note that for any $f \in \mathcal{M}_d$, 
	\begin{align*}
	\abs{(\Prob_n - P)f} &= \bigabs{\E_{\Prob_n}f(X)-\E_P f(X)} \\
	&\leq \bigabs{\E_{\Prob_n}f_+(X)-\E_P f_+(X)} + \bigabs{\E_{\Prob_n}f_-(X)-\E_P f_-(X)}\\
	&\leq \biggabs{\int_0^\infty  \big(\Prob_{\Prob_n}(f_+(X)> t)- \Prob_{P}(f_+(X)> t)\big)\ \d{t} }\\
	&\qquad +  \biggabs{\int_0^\infty  \big(\Prob_{\Prob_n}(f_-(X)> t)- \Prob_{P}(f_-(X)> t)\big)\ \d{t} }\\
	&\leq 2 \pnorm{f}{\infty} \sup_{C \in \mathcal{L}_d} \abs{(\Prob_n-P) (C)}.
	\end{align*}
	Here $\mathcal{L}_d$ is the class of all upper and lower sets in $[0,1]^d$. The last inequality follows since for any $f \in \mathcal{M}_d$, $\{f_+(x) >t\}=\{f(x) \vee 0>t\} \in \mathcal{L}_d$ and $\{f_-(x)>t\} = \{-(f(x)\wedge 0)>t\}=\{f(x)\wedge 0<-t\} \in \mathcal{L}_d$. Hence by \cite[Theorem 8.22]{dudley1999uniform}, we may apply Theorem \ref{thm:upper_bound_ep_set} with $\alpha = d-1$ to see that
	\begin{align*}
	\E \sup_{f \in \mathcal{M}_d \cap L_\infty(C\sqrt{\log n})} \abs{\G_n(f-f_0)}\lesssim \sqrt{\log n} \cdot \E \sup_{C \in \mathcal{L}_d} \abs{\G_n(C)} + 1\lesssim \sqrt{\log n}\cdot n^{\frac{d-2}{2d}}.
	\end{align*}
	Using the multiplier inequality (cf. Lemma \ref{lem:multiplier_ineq}) and contraction principle for empirical processes (cf. \cite[Lemma 6, supplement]{han2017isotonic}), we may further bound the two empirical processes in (\ref{ineq:iso_reg_1}) by
	\begin{align*}
	&\E \sup_{ \substack{f \in \mathcal{M}_d\cap L_\infty(C \sqrt{\log n}):\\ \pnorm{f-f_0}{L_2(P)}\leq \bar{r}_n}} \abs{\G_n(\xi (f-f_0))}\\
	&\qquad \bigvee \E \sup_{ \substack{f \in \mathcal{M}_d\cap L_\infty(C \sqrt{\log n}):\\ \pnorm{f-f_0}{L_2(P)}\leq \bar{r}_n}} \abs{\G_n((f-f_0)^2)}\lesssim  n^{\frac{d-2}{2d}} \log n.
	\end{align*}
	Solving (\ref{ineq:iso_reg_1}) using the above inequality we obtain the rate $\bar{r}_n$ for $d\geq 3$.
	
	For $d=2$, we may estimate the empirical process with an additional $\log n$:
	\begin{align*}
	\E \sup_{f \in \mathcal{M}_d \cap L_\infty(C\sqrt{\log n})} \abs{\G_n(f-f_0)}\lesssim \log^{3/2} n,
	\end{align*}
	and the rate can be obtained similarly as above.
\end{proof}

\begin{remark}\label{rmk:isotonic_reg}
	The proof for the analogue of Theorem \ref{thm:iso_reg} in \cite{han2017isotonic}, i.e. \cite[Theorem 4]{han2017isotonic}, uses a completely different strategy. A rough argument is as follows. \cite{han2017isotonic} first consider the problem $f_0=0$, where it is shown in Proposition 9 therein that for $\delta_n>0$ not too small,
	\begin{align}\label{ineq:iso_reg_HWCS_1}
	\E \sup_{f \in \mathcal{M}_d \cap L_\infty(\sqrt{C\log n}): \pnorm{f}{L_2(P)}\leq \delta_n} \abs{\G_n(f)}\lesssim \delta_n \cdot n^{1/2-1/d} \log^\gamma n.
	\end{align}
	Then by a simple triangle inequality, if $d\geq 2$,
	\begin{align}\label{ineq:iso_reg_HWCS_2}
	&\E \sup_{ \substack{ f \in \mathcal{M}_d \cap L_\infty(C\sqrt{\log n}):\\ \pnorm{f-f_0}{L_2(P)}\leq \delta_n}} \abs{\G_n(f-f_0)}\\
	&\leq \E \sup_{ \substack{f \in \mathcal{M}_d\cap L_\infty(C\sqrt{\log n}):\\ \pnorm{f}{L_2(P)}\leq \delta_n+\pnorm{f_0}{\infty}}} \abs{\G_n(f)} + \E\abs{\G_n(f_0)}\nonumber\lesssim (\delta_n+\pnorm{f_0}{\infty})n^{1/2-1/d} \log^{\gamma} n.\nonumber
	\end{align}
	Using the above inequality and (\ref{ineq:iso_reg_1}), we obtain $\bar{r}_n^2 \lesssim n^{-1/d}$ up to logarithmic factors. It is clear from the sketch here that the property of isotonic regression functions is only used in (\ref{ineq:iso_reg_HWCS_1}) where the problem is $f_0=0$. The proof for general $f_0 \in L_\infty(1)$ in (\ref{ineq:iso_reg_HWCS_2}) is not very informative in the sense that the method of (\ref{ineq:iso_reg_HWCS_2}) is valid for any problem as long as one could solve the risk problem (= empirical process problem (\ref{ineq:iso_reg_HWCS_1})) for one particular $f_0$. In contrast, the proof of Theorem \ref{thm:iso_reg} here shows that it is the complexity of the class of upper and lower sets $\mathcal{L}_d$ that leads to the minimax rate of convergence for the multiple isotonic LSE.
\end{remark}

\begin{remark}\label{rmk:conv_reg}
		It is possible to adapt the present approach to the problem of multivariate convex regression. The major difficulty here is to understand the boundary behavior for the convex LSE $\widehat{f}_n^{\mathrm{cvx}}$. In particular, if we can prove that the convex LSE $\widehat{f}_n^{\mathrm{cvx}}$ satisfies $\pnorm{\widehat{f}_n^{\mathrm{cvx}}-f_0}{\infty}=\mathcal{O}_{\mathbf{P}}(L_n)$ for some slowly growing $L_n$ (in similar spirit to Lemma \ref{lem:trunc_iso_reg} for the isotonic LSE), then using similar arguments as in the proof of Theorem \ref{thm:iso_reg}, we may conclude $\pnorm{\widehat{f}_n^{\mathrm{cvx}}-f_0}{L_2(P)}= \mathcal{O}_{\mathbf{P}}(n^{-1/(d+1)}\bar{L}_n)$ for some slowly growing $\bar{L}_n$. Interestingly, recently \cite{kur2020convex} proved that under a fixed lattice design in a polytopal domain, the rate of the convex LSE can be improved to $n^{-2/(d+4)}$ for $d\leq 4$ and $n^{-1/d}$ for $d\geq 5$ (up to logarithmic factors) and these rates cannot be further improved in the worst case. These improved rates are due to reduced complexity of the class of convex functions on polytopal domains than those on smooth domains. Similar rates are obtained in \cite{kur2020convex} for bounded and Lipschitz convex LSEs under random designs. It remains open whether the unconstrained convex LSEs attain these rates over polytopal and more general domains under random designs, where the boundary behavior of $\widehat{f}_n^{\mathrm{cvx}}$ may play a crucial role.
\end{remark}

\subsection{Proof of Theorem \ref{thm:rate_s_concave}}

\begin{proof}[Proof of Theorem \ref{thm:rate_s_concave}]
	We only provide the proof for the most difficult case $-1/d<s<0$; the other cases are similar or simpler. 
	
	\noindent \textbf{(Case 1: $d\geq 4$).} We will relate the squared Hellinger distance $h^2(p_0,\widehat{p}_n)$ to that of the expected supremum of empirical process over the class of convex sets. The proof for this reduction is largely inspired by the idea of \cite{carpenter2018near}.
	
	Using same arguments as in Step 1 of the proof of \cite[Theorem 4.3]{doss2013global}, we may assume without loss of generality that $p_0 \in \mathcal{P}_{s,M/2}$ and $\widehat{p}_n$ belongs to
	\begin{align*}
	\mathcal{P}_{s,M}\equiv \bigg\{p \in \mathcal{P}_s: \sup_{x \in \R^d} p(x)\leq M, \inf_{x \in B(0,1)} p(x)\geq 1/M\bigg\}
	\end{align*}
	for some large $M$ with high probability. By the proof of  \cite[Lemma F.7]{han2015approximation} (especially (F.3) therein),
	\begin{align}\label{ineq:s_concave_0}
	\sup_{p \in \mathcal{P}_{s,M}}p(x)\leq C_{M} (1+\pnorm{x}{})^{1/s}\leq C_{M,d} \bigg(1+\prod_{k=1}^d \abs{x_k}^{1/d}\bigg)^{1/s}.
	\end{align}
	Furthermore, it is not hard to see that $\widehat{p}_n$ is supported in the convex hull of $X_1,\ldots,X_n$. 	By (\ref{ineq:s_concave_0}), $\kappa_q\equiv \E_{X\sim p_0} (1/p_0(X))^q = \int p_0^{1-q}<\infty$ for $q \in (0,1+sd)$. This means that $\E_{X \sim p_0}  \max_{1\leq i\leq n} (1/p_0(X_i))^q \leq n\cdot \kappa_q$, so $\log \max_i (1/p_0(X_i)) \leq C_1 \log n$ with high probability for large $C_1>0$.	Hence with $c_n\equiv n^{-C_1}$, $X_1,\ldots,X_n \in \{p_0\geq c_n\}$ with high probability. 
	Let $\widetilde{p}_n\equiv (\widehat{p}_n\vee c_n)\bm{1}_{p_0\geq c_n} / \int (\widehat{p}_n\vee c_n)\bm{1}_{p_0 \geq c_n}$. Then with $b_n \equiv \int (\widehat{p}_n \vee c_n)\bm{1}_{p_0\geq c_n}$, it follows that with high probability
	\begin{align*}
	b_n^{-1} &= b_n^{-1} \int_{p_0\geq c_n} \widehat{p}_n \leq b_n^{-1}\int_{p_0\geq c_n} (\widehat{p}_n \vee c_n)=1,\\
	b_n-1 &= \int_{p_0\geq c_n} \abs{(\widehat{p}_n \vee c_n)-\widehat{p}_n} \leq c_n \abs{\{p_0\geq c_n\}}\lesssim c_n^{1+sd}.
	\end{align*}
	The last inequality in the second line of the above display follows as $\{x: \pnorm{x}{}>(c_n/C_M)^s\}\subset \{x: p_0(x)<c_n\}$ by (\ref{ineq:s_concave_0}), and therefore $\abs{\{p_0\geq c_n\}} \subset \abs{\{x: \pnorm{x}{}\leq (c_n/C_M)^s\}} \asymp_{d, M} c_n^{sd}$. As $s>-1/d$, by choosing $C_1>0$ large, we have $0\leq b_n-1\leq \mathcal{O}(n^{-1})$. This implies with high probability,
	\begin{align}\label{ineq:s_concave_0a}
	h^2(\widehat{p}_n,\widetilde{p}_n)&\lesssim \int \abs{\widehat{p}_n-\widetilde{p}_n} = \int \abs{\widehat{p}_n- (\widehat{p}_n\vee c_n)\bm{1}_{p_0\geq c_n} /b_n} \nonumber\\
	&\leq \abs{1-b_n^{-1}}\int \widehat{p}_n + b_n^{-1} \int_{p_0\geq c_n} \abs{\widehat{p}_n - (\widehat{p}_n\vee c_n)} \nonumber\\
	&= \abs{1-b_n^{-1}} + b_n^{-1} (b_n-1) = \mathcal{O}(n^{-1}).
	\end{align}
	
	On the other hand, let $\widetilde{p}_0\equiv  p_0 \bm{1}_{p_0\geq c_n}/ \int p_0 \bm{1}_{p_0\geq c_n}$ and $b_0\equiv \int p_0 \bm{1}_{p_0\geq c_n} $. As $s>-1/d$, with $q \in (0,1+sd)$, 
	\begin{align*}
	b_0&= \int p_0 \bm{1}_{p_0\geq c_n} \leq 1,\\
	1-b_0& =\int p_0 \bm{1}_{p_0< c_n} \lesssim c_n^q \int (1+\pnorm{x}{})^{(1-q)/s}\,\d{x} =\mathcal{O}(c_n^q). 
	\end{align*}
	This means by choosing $C_1>0$ large, we have $0\leq 1-b_0\leq \mathcal{O}(n^{-1})$. Hence
	\begin{align}\label{ineq:s_concave_0b}
	h^2(p_0,\widetilde{p}_0)& \asymp \int p_0 \bm{1}_{p_0< c_n}+ \int \big(\sqrt{p_0}-\sqrt{p_0/b_0}\big)^2 \bm{1}_{p_0\geq c_n}  \nonumber\\
	& = \mathcal{O}(c_n^q)+ \big(1-b_0^{-1/2}\big)^2 \int p_0 \bm{1}_{p_0\geq c_n} = \mathcal{O}(n^{-1}),
	\end{align}
	and
	\begin{align}\label{ineq:s_concave_0c}
	&\biggabs{\int p_0\log p_0 - \int \widetilde{p}_0 \log \widetilde{p}_0} \nonumber\\
	&\leq \biggabs{\int p_0\log p_0 \bm{1}_{p_0<c_n}}+  \biggabs{\int_{p_0\geq c_n} p_0 \log p_0 - (p_0/b_0)\log (p_0/b_0)} \nonumber \\
	&\leq  \mathcal{O}(c_n^q)+ \bigabs{1-b_0^{-1}} \int_{p_0\geq c_n} \abs{p_0 \log p_0} + \bigabs{\log b_0/b_0} \int_{p_0\geq c_n} p_0 \nonumber \\
	&= \mathcal{O}(n^{-1}),
	\end{align}
	and
	\begin{align}\label{ineq:s_concave_0d}
    \int \abs{p_0-\widetilde{p}_0} \leq \int p_0 \bm{1}_{p_0<c_n}+ \abs{1-b_0^{-1}} \int_{p_0\geq c_n} p_0 = \mathcal{O}(n^{-1}).
	\end{align}
	Now by the integrability $\E_{P_0}\log^2 p_0<\infty$, with $P_0, \widetilde{P}_0$ denoting the distributions of $p_0,\widetilde{p}_0$, it follows that with high probability,
	\begin{align}\label{ineq:s_concave_01}
	h^2(p_0,\widehat{p}_n)	&\lesssim h^2(p_0,\widetilde{p}_0)+ h^2(\widetilde{p}_0,\widetilde{p}_n)+ h^2(\widetilde{p}_n,\widehat{p}_n)  \nonumber \\
	&\leq \E_{\widetilde{P}_0} \log (\widetilde{p}_0/\widetilde{p}_n) + \mathcal{O}(n^{-1}) \quad \hbox{(by (\ref{ineq:s_concave_0a}) and (\ref{ineq:s_concave_0b}))}\nonumber \\
	&\leq \E_{P_0} \log p_0 - \E_{\widetilde{P}_0} \log \tilde{p}_n + \mathcal{O}(n^{-1}) \quad \hbox{(by (\ref{ineq:s_concave_0c}))} \nonumber \\
	&\leq \E_{\Prob_n} \log p_0 - \E_{\widetilde{P}_0}\log \widetilde{p}_n + \mathcal{O}_{\mathbf{P}}(n^{-1/2}) \quad (\textrm{by integrability of }\log p_0)\nonumber \\
	&\leq \E_{\Prob_n} \log \widehat{p}_n - \E_{\widetilde{P}_0}\log \widetilde{p}_n + \mathcal{O}_{\mathbf{P}}(n^{-1/2})\quad  (\textrm{as }\widehat{p}_n\textrm{ is the MLE}) \nonumber\\
	&\leq \E_{\Prob_n} \log \big[b_n^{-1}(\widehat{p}_n \vee c_n)\bm{1}_{p_0\geq c_n}\big] - \E_{\widetilde{P}_0}\log \widetilde{p}_n + \log b_n+\mathcal{O}_{\mathbf{P}}(n^{-1/2})\nonumber\\
	&\leq\abs{ (\Prob_n-\widetilde{P}_0) \log \widetilde{p}_n}+ \mathcal{O}_{\mathbf{P}}(n^{-1/2}\vee \abs{1-b_n})\nonumber\\
	& \leq \biggabs{\int_0^\infty  \big(\Prob_{\Prob_n}((\log \widetilde{p}_n)_+(X)\geq t)- \Prob_{\widetilde{P}_0}((\log \widetilde{p}_n)_+(X)\geq  t)\big)\ \d{t} }\nonumber\\
	&\qquad +  \biggabs{\int_0^\infty  \big(\Prob_{\Prob_n}((\log \widetilde{p}_n)_-(X)\leq t)- \Prob_{\widetilde{P}_0}((\log \widetilde{p}_n)_-(X)\leq t)\big)\ \d{t} }\nonumber\\
	&\qquad + \mathcal{O}_{\mathbf{P}}(n^{-1/2})\nonumber\\
	&\stackrel{(*)}{\lesssim} \log n\cdot \E \sup_{C \in \mathscr{C}_d} \abs{(\Prob_n-\widetilde{P}_0) (C)}+ \mathcal{O}_{\mathbf{P}}(n^{-1/2})\nonumber\\
	&\leq \log n\cdot \bigg[\E \sup_{C \in \mathscr{C}_d} \abs{(\Prob_n-P_0) (C)}+ \sup_{C \in \mathscr{C}_d} \int_C \abs{p_0-\widetilde{p}_0}\bigg]+ \mathcal{O}_{\mathbf{P}}(n^{-1/2})\nonumber\\
	& \stackrel{(\ast\ast)}{\leq} \log n\cdot \E \sup_{C \in \mathscr{C}_d} \abs{(\Prob_n-P_0) (C)}+ \mathcal{O}_{\mathbf{P}}(n^{-1/2}).
	\end{align}
	Here $\mathscr{C}_d $ is the set of all convex bodies in $\R^d$. The inequality ($\ast$) follows as for any $s$-concave density $p$, $\{(\log p(x))_+\geq t\} =\{ \log p(x) \vee 0\geq t \} = \{p(x)\vee 1\geq e^t\}= \{\varphi(x) \wedge 1 \leq e^{st}\}$ and $\{(\log p(x))_-\leq t\} =\{ -\big(\log p(x) \wedge 0\big)\leq t \} = \{p(x)\wedge 1\geq e^{-t}\}= \{\varphi(x) \vee 1 \leq e^{-ts}\}$  are convex sets, and $-C_1\log n\leq \log c_n \leq \log \widetilde{p}_n \leq \log (M)$ for $n$ large. The inequality ($\ast\ast$) follows from (\ref{ineq:s_concave_0d}).

	Hence we only need to bound the entropy $\mathcal{N}_I(\epsilon, \mathscr{C}_d, P_0)$. To this end, for a multi-index $\ell = (\ell_1,\ldots,\ell_d) \in \mathbb{Z}_{\geq 0}^d$, let $I_\ell \equiv \prod_{k=1}^d [2^{\ell_k}-1,2^{\ell_k+1}-1]$. Then $\abs{I_\ell}\asymp 2^{\sum_k \ell_k}$. Let $\{(A_{\ell, j}, B_{\ell,j}):1\leq j\leq N_\ell\}$ be an $\epsilon_{\ell}$-bracket for $\{C|_{I_\ell}:C \in \mathscr{C}\}$ under the Lebesgue measure. By \cite[Theorem 8.25]{dudley1999uniform}, we have $\log N_\ell \lesssim_d (\abs{I_\ell}^{-1}\epsilon_\ell)^{(1-d)/2}$. Let $\epsilon_\ell = a_\ell\cdot \epsilon$, where $a_\ell \equiv \abs{I_\ell}^{(1+\delta)}$ for some $\delta>0$ such that $1<1+\delta<(-sd)^{-1}\in (1,\infty)$. Then $\{(\sum_{\ell} \bm{1}_{A_{\ell, j_\ell}}\bm{1}_{I_\ell}, \sum_{\ell}
	\bm{1}_{B_{\ell,j_\ell}}\bm{1}_{I_\ell}):1\leq j_\ell\leq N_\ell, \ell \in \mathbb{Z}_{\geq 0}^d\}$ forms a bracket for $\mathscr{C}\cap \R_{\geq 0}^d$ with $P_0$-size
	\begin{align*}
	\biggabs{P_0\bigg(\sum_{\ell}
		\bm{1}_{B_{\ell,j_\ell}}\bm{1}_{I_\ell}-\sum_{\ell}
		\bm{1}_{A_{\ell,j_\ell}}\bm{1}_{I_\ell}\bigg)}&\leq \sum_{\ell} \epsilon_\ell \sup_{x \in I_\ell} p_0(x)\lesssim \epsilon \sum_{\ell} a_\ell  \abs{I_\ell}^{-(-sd)^{-1}}\lesssim \epsilon.
	\end{align*}
	The logarithm of the number of the brackets can be bounded by
	\begin{align*}
	C_d \sum_{\ell} \abs{I_\ell}^{(d-1)/2} \epsilon_\ell^{(1-d)/2} \lesssim \epsilon^{(1-d)/2} \sum_{\ell}  \big(a_\ell  \abs{I_\ell}^{-1} \big)^{(1-d)/2}\lesssim \epsilon^{(1-d)/2}.
	\end{align*}
	Other quadrants can be handled similarly. This means that $\log \mathcal{N}_I(\epsilon, \mathscr{C},P_0)\lesssim \epsilon^{(1-d)/2}$, and hence Theorem \ref{thm:upper_bound_ep_set} applies to (\ref{ineq:s_concave_01}). 
	
	\noindent \textbf{(Case 2: $d=3$).}  The situation for $d=3$ is similar; but with an additional $\log n$ term in the estimate for the empirical process $\E \sup_{C \in \mathscr{C}_3} \abs{\G_n(C)} \lesssim \log n$ (cf. Remark \ref{rmk:upper_bound_entropy}), and therefore the rate in squared Hellinger comes with an additional $\log n$.
	
	\noindent \textbf{(Case 3: $d=2$).}  We employ an idea in \cite{doss2013global} in $d=1$, which first calculates the $L_2$ entropy of the class of bounded $s$-concave functions on $[0,1]$ by discretization of the range of the underlying convex functions until a prescribed $L_2$ error is reached at the level of $s$-concave densities, and then use the integrability of $\mathcal{P}_{s,M}$ to extend the brackets from $[0,1]$ to $\R$. Our arguments below substantially simplify those presented in \cite{doss2013global}. A similar idea is exploited in \cite{kim2016global} in $d=2,3$ in the context of log-concave densities, with further technicalities due to the unknown shapes of domain at the truncated levels in dimensions $2$ and $3$ (for $d=1$ they are simply intervals). 
	
	Let $\widetilde{\mathcal{P}}_s(I, B) \equiv \{p\textrm{ is s-concave on }I \subset \R^2: 0\leq p(x)\leq B, \forall x \in I\}$. We write $\widetilde{\mathcal{P}}_s=\widetilde{\mathcal{P}}_s([0,1]^2,1)$ for simplicity. We claim that for $s>-1$,
	\begin{align}\label{ineq:s_concave_1}
	\log \mathcal{N}_{[\,]}(\epsilon, \widetilde{\mathcal{P}}_s, L_2)\lesssim_s \epsilon^{-1}\log(1/\epsilon).
	\end{align}
	Fix $\epsilon>0$. Let $y_k\equiv 2^{k}, 1\leq k\leq k_0$, where $k_0$ is the smallest integer such that $y_{k_0}^{1/s}\leq \epsilon$, i.e. $k_0 = \lceil \log_2((1/\epsilon)^{-s})\rceil$. Let $\{(A_j,B_j): A_j \supset B_j\}_{j=1}^{N_1}$ be an $\epsilon_0^2$-bracket for all convex sets in $[0,1]^2$ under the Lebesgue measure. By \cite[Theorem 8.25, Corollary 8.26]{dudley1999uniform}, we have $\log N_1 \lesssim \epsilon_0^{-1}$. By \cite[Proposition 4, supplement]{kim2016global}, for each $j=1,\ldots,N_1$, and $k=1,\ldots,k_0$, we may find a lower $\epsilon_{j,k}$-bracket $\{\underline{f}_{j,k,m}:1\leq m\leq \underline{N}_{j,k}\}$ in $L_2$ (resp. upper $\epsilon_{j,k}$-bracket $\{\bar{f}_{j,k,m}:1\leq m\leq \bar{N}_{j,k}\}$ in $L_2$) for non-negative convex functions defined on $B_j$ with an upper bound $2^k$, such that $\log (\bar{N}_{j,k}\vee \underline{N}_{j,k})\lesssim (2^k/\epsilon_{j,k}) \log(2^k/\epsilon_{j,k})$. 
	
	For any $p \in \widetilde{\mathcal{P}}_s$, let $\varphi = p^s$ be the underlying convex function. Let $C_k\equiv \{\varphi \leq y_k\}$. 
	Let $(A_{j_k},B_{j_k}), A_{j_k}\supset B_{j_k}$ be a bracket for $A_{j_k}\supset C_k\supset B_{j_k}$, and let $\underline{f}_{j_k,k,m}$ (resp. $\bar{f}_{j_k,k,m}$) be a lower (resp. upper) bracket for ${\varphi}|_{B_{j_k}}$.
	
	Let $A_{j_0}\equiv \emptyset $ and $y_0\equiv  1$. Consider an upper bracket for $p$ of form 
	\begin{align*}
	\sum_{k=1}^{k_0} \bigg[ \big(\underline{f}_{j_k,k,m}\vee y_{k-1}\big)\bm{1}_{B_{j_k}\setminus A_{j_{k-1}}}\bigg]^{1/s}+ \sum_{k=1}^{k_0} \big[(y_{k-1})^{1/s}\wedge 1\big] \bm{1}_{A_{j_k}\setminus B_{j_k}}+\epsilon \bm{1}_{[0,1]^2\setminus A_{j_{k_0}}},
	\end{align*}
	and a lower bracket of $p$ of form $
	\sum_{k=1}^{k_0} \big[ \big(\bar{f}_{j_k,k,m}\wedge y_{k}\big)\bm{1}_{B_{j_k}\setminus A_{j_{k-1}}}\big]^{1/s}$. For the choice $
	\epsilon_0\equiv \epsilon$ and $\epsilon_{j,k}\equiv \epsilon \cdot 2^{2k}$, this bracket has squared $L_2$ size bounded, up to a constant depending only on $s$, by 
	\begin{align*}
	\sum_{k=1}^{k_0} \epsilon_{j_k,k}^2\cdot (2^k)^{2(1/s-1)}+ \epsilon_0^2 \cdot \sum_{k=1}^{k_0} \big[(y_{k-1})^{1/s}\wedge 1\big] +\epsilon^2\lesssim \epsilon^2.
	\end{align*}
	The logarithm of the total number of brackets can be bounded by
	\begin{align*}
	\log \bigg[\prod_{k=1}^{k_0} N_1^2 \bar{N}_{j_k,k} \underline{N}_{j_k,k}\bigg] \lesssim \sum_{k=1}^{k_0} \bigg(\epsilon_0^{-1}+\frac{2^k}{\epsilon_{j_k,k}}\log\bigg(\frac{2^k}{\epsilon_{j_k,k}}\bigg) \bigg)\lesssim \epsilon^{-1}\log(1/\epsilon),
	\end{align*}
	proving the claim (\ref{ineq:s_concave_1}). Let $I_\ell$ be the same as in the proof for $d\geq 4$. By rescaling, it follows that 
	\begin{align*}
	\log \mathcal{N}_{[\,]}(\epsilon,\widetilde{\mathcal{P}}_s(I_\ell, B), L_2)\lesssim \frac{(B^2\abs{I_\ell})^{1/2}}{\epsilon} \log\bigg(\frac{(B^2\abs{I_\ell})^{1/2}}{\epsilon}\bigg).
	\end{align*}
	By (\ref{ineq:s_concave_0}), on $I_\ell$, $\sup_{x \in I_\ell}\sup_{p \in \mathcal{P}_{s,M}} p(x)\leq \abs{I_\ell}^{1/2s}$. Let $b_\ell = \abs{I_\ell}^{-\delta'}$ for some $\delta' \in (0,(-1/s-1)/2)$, and $\{\underline{f}_{j,\ell},\bar{f}_{j,\ell}:1\leq j\leq N_\ell\}$ be a $b_\ell \epsilon$-bracket for $\mathcal{P}_{s,M}|_{I_\ell}$ under $L_2$. A global bracket for $\mathcal{P}_{s,M}$ can be obtained by assembling these local brackets for all(=four) quadrants, with squared $L_2$-size at most $\epsilon^2 \sum_{\ell} b_\ell^2\lesssim \epsilon^2 $, and the logarithm of the number of brackets is
	\begin{align*}
	\sum_{\ell} \log N_\ell \lesssim \sum_{\ell} \frac{\abs{I_\ell}^{(1/s+1)/2}}{b_\ell \epsilon}\log \bigg(\frac{\abs{I_\ell}^{(1/s+1)/2}}{b_\ell \epsilon}\bigg)\lesssim \epsilon^{-1}\log(1/\epsilon).
	\end{align*}
	Hence for $s>-1/2$, $
	\log \mathcal{N}_{[\,]}(\epsilon, \mathcal{P}_{s,M}, h)= \log \mathcal{N}_{[\,]}(\epsilon, \mathcal{P}_{2s,M}, L_2)\lesssim \epsilon^{-1}\log(1/\epsilon)$. The rest of the proof is a standard computation of the size of the localized empirical process via Hellinger bracketing numbers (cf. \cite[Theorem 3.4.4]{van1996weak}), so we omit the details.
\end{proof}

\appendix

\section{Upper and lower bounds for weighted empirical processes}\label{section:ratio_ep}

As a direct application of Theorem \ref{thm:upper_bound_ep_set}, we consider upper and lower bounds for ratio-type empirical processes. Such bounds are initiated in \cite{wellner1978limit,shorack1982limit,stute1982oscillation,mason1983strong,stute1984oscillation} for uniform empirical processes on (subsets of) $\R$ (or $\R^d$), and are further investigated in \cite{alexander1987rates} for VC classes of sets, and extended by \cite{gine2003ratio,gine2006concentration} who studied more general VC-subgraph classes. These authors work with classes satisfying uniform entropy conditions,  and the class of sets (or functions) need to be Donsker apriori. The lack of corresponding results for non-Donsker class of sets are mainly due to the lack of sharp upper and lower bounds for the behavior of the empirical process. Here we fill in this gap by using Theorem \ref{thm:upper_bound_ep_set}.

\begin{theorem}\label{thm:ratio_ep_std}
	Let $r_n^2\gtrsim n^{-1/(\alpha+1)}$ and $
	\gamma_n \equiv n^{1/2} r_n \big(r_n^{\alpha-1} \wedge n^{-\frac{\alpha-1}{2(\alpha+1)}}\big)$.
	Then we have the following:
	\begin{enumerate}
		\item If (E1) holds and $r_n^{2\alpha} \log \log n\to 0$,
		\begin{align*}
		\limsup_{n \to \infty}\gamma_n \sup_{C \in \mathscr{C}: r_n^2\leq P(C)\leq 1} \frac{\abs{\Prob_n(C)-P(C)}}{\sqrt{P(C)}}<\infty\qquad \textrm{a.s.}
		\end{align*}
		\item 
		If (E1)-(E2) hold, 
		\begin{align*}
		\liminf_{n \to \infty} \gamma_n \sup_{C \in \mathscr{C}: r_n^2\leq P(C)\leq 1} \frac{\abs{\Prob_n(C)-P(C)}}{\sqrt{P(C)}}>0\qquad \textrm{a.s.}
		\end{align*}
	\end{enumerate}
\end{theorem}

\begin{theorem}\label{thm:ratio_ep_var}
	Let $r_n^2\gtrsim n^{-1/(\alpha+1)}$ and there exists some large constant $K_\alpha>0$ such that: 
	\begin{enumerate}
		\item If (E1) holds and $
		\liminf_{n \to \infty} r_n^2 \cdot n^{1/(\alpha+1)}\geq \underline{\rho}$
		for some $\underline{\rho} \in (K_\alpha,\infty]$, then 
		\begin{align*}
		\limsup_{n \to \infty} \sup_{C \in \mathscr{C}: r_n^2\leq P(C)\leq 1} \biggabs{\frac{\Prob_n(C)}{P(C)}-1} \leq \mathcal{O}\left(\underline{\rho}^{-\left( 1\wedge \frac{1+\alpha}{2}\right) }\right)\qquad \textrm{a.s.}
		\end{align*}
		\item If furthermore (E2) holds and $
		\limsup_{n \to \infty} r_n^2 \cdot n^{1/(\alpha+1)}\leq \bar{\rho}$
		for some $\bar{\rho} \in (K_\alpha,\infty]$, then 
		\begin{align*}
		\liminf_{n \to \infty} \sup_{C \in \mathscr{C}: r_n^2\leq P(C)\leq 1} \biggabs{\frac{\Prob_n(C)}{P(C)}-1} \geq \mathcal{O}\left(\bar{\rho}^{-\left( 1\wedge \frac{1+\alpha}{2}\right) }\right)\qquad \textrm{a.s.}
		\end{align*}
		
	\end{enumerate}
\end{theorem}

\begin{remark}\label{rmk:ratio_ep}
	Some technical remarks:
	\begin{enumerate}
		\item An interesting corollary of Theorem \ref{thm:ratio_ep_var} is that under entropy conditions (E1)-(E2), the sequence  in the theorem converges to $0$ as $n \to \infty$ almost surely if and only if $r_n^2\cdot n^{1/(\alpha+1)}\to \infty$. 
		\item Theorems \ref{thm:ratio_ep_std} and \ref{thm:ratio_ep_var} are also valid in their $L_p (1\leq p<\infty)$ versions (which can be seen by integrating the tail estimates in the proofs). For instance, if (E1) holds, then
		\begin{align*}
		\limsup_{n \to \infty} \biggpnorm{\gamma_n \sup_{C \in \mathscr{C}: r_n^2\leq P(C)\leq 1} \frac{\abs{\Prob_n(C)-P(C)}}{\sqrt{P(C)}}}{L_p(P^{\otimes n})}<\infty,
		\end{align*}
		and 
		\begin{align*}
		\limsup_{n \to \infty} \biggpnorm{\sup_{C \in \mathscr{C}: r_n^2\leq P(C)\leq 1} \biggabs{\frac{\Prob_n(C)}{P(C)}-1}}{L_p(P^{\otimes n})} \leq \mathcal{O}\left(\underline{\rho}^{-\left( 1\wedge \frac{1+\alpha}{2}\right) }\right).
		\end{align*}
		\item 	We may consider more general weighting functions of form $\phi(\sqrt{P(C)})$ as in \cite{gine2006concentration} rather than the special cases $\phi_1(t)=t$ in Theorem \ref{thm:ratio_ep_std} and $\phi_2(t)=t^2$ in Theorem \ref{thm:ratio_ep_var}. Here we make these choices mainly due to the fact that $\phi_1,\phi_2$ are of special interest in the history of empirical process theory \cite{wellner1978limit,shorack1982limit,stute1982oscillation,mason1983strong,stute1984oscillation,alexander1987rates}, and the corresponding results for more general cases follow from minor modifications of the proofs.
		\item It is also straightforward to consider corresponding ratio limit theorems for function classes satisfying the conditions of Theorem \ref{thm:upper_bound_ep}; we omit these digressions.
	\end{enumerate}
	
\end{remark}

We will investigate the behavior of ratio-type empirical processes in a more general setting as in \cite{gine2006concentration}. Let $\phi$ be a continuous and strictly increasing function with $\phi(0)=0$. Let $\mathscr{C}(r)\equiv \{C \in \mathscr{C}: P(C)\leq r^2\}$ and $\mathscr{C}(r,s] \equiv \mathscr{C}(s)\setminus \mathscr{C}(r)$. Fix $0<r<\delta\leq 1$. For a real number $1<q\leq 2$, let $\ell \equiv \ell_{r,\delta,q}$ be the smallest integer no smaller than $\log_q(\delta/r)$. For any $\bm{s}\equiv (s_1,\ldots,s_\ell) \in \R_{\geq 0}^\ell$, let
\begin{align*}
\beta_{n,q}(r,\delta)&\equiv \max_{1\leq j\leq \ell} \frac{\E \sup_{C \in \mathscr{C}(rq^{j-1},rq^j]}\abs{\G_n(C)} }{\phi(rq^j)}, \tau_{n,q}(r,\delta,\bm{s})\equiv \max_{1\leq j\leq \ell} \frac{rq^j \sqrt{ s_j}+s_j/\sqrt{n}}{\phi(rq^j)}.\nonumber
\end{align*}
The following result is essentially due to \cite{gine2006concentration}. We state a somewhat simplified and easier-to-use version. 

\begin{proposition}\label{prop:deviation_normalized_process}
	Assume that $\phi$ is continuous, strictly increasing and satisfies $
	\sup_{r\leq x\leq 1} \phi(qx)/\phi(x)=\kappa_{r,q}<\infty$
	for some $1<q\leq 2$. Then for any $\bm{s}\equiv (s_1,\ldots,s_\ell) \in \R_{\geq 0}^\ell$, both the probabilities
	\begin{align*}
	&\Prob\bigg[\sup_{C \in \mathscr{C}: r^2<P(C)\leq \delta^2} \frac{\abs{\G_n(C)}}{\phi(\sqrt{P(C)})}\geq K\kappa_{r,q}\big(\beta_{n,q}(r,\delta)+ \tau_{n,q}(r,\delta,\bm{s})\big)\bigg]
	\end{align*}
	and
	\begin{align*}
	&\Prob\bigg[\sup_{C \in \mathscr{C}: r^2<P(C)\leq \delta^2} \frac{\abs{\G_n(C)}}{\phi(\sqrt{P(C)})}\leq  K\big(\beta_{n,q}(r,\delta)- \tau_{n,q}(r,\delta,\bm{s})\big)\bigg]
	\end{align*}
	can be bounded by $
	K \sum_{j=1}^\ell \exp\big(-s_j/K\big)$. Here $K>0$ is a universal constant.
\end{proposition}

\begin{proof}[Proof of Proposition \ref{prop:deviation_normalized_process}]
	We only prove the first claim; the second follows from similar arguments. The proof is a simple application of Talagrand's concentration inequality combined with a peeling device. Write $\mathscr{C}_j\equiv \mathscr{C}(rq^{j-1},rq^j]$ and $\phi_q(u)\equiv \phi(rq^j)$ if $u \in (rq^{j-1},rq^j]$ for notational convenience. By Talagrand's concentration inequality,
	\begin{align*}
	\Prob\bigg[\sup_{C \in \mathscr{C}_j}\abs{\G_n(C)}\geq K\bigg(\E \sup_{C \in \mathscr{C}_j}\abs{\G_n(C)} + \sqrt{\sigma^2_j s_j}+\frac{s_j}{\sqrt{n}}\bigg)\bigg]\leq K\exp\big(-s_j/K\big)
	\end{align*}
	where $\sigma_j^2 = \sup_{f \in \mathscr{C}_j} P(C)=r^2q^{2j}$. Hence by a union bound we see that with probability at least $1-\sum_{j=1}^\ell  K\exp(-s_j/K)$, it holds that
	\begin{align*}
	&\bigg(\sup_{C \in \mathscr{C}: r^2<P(C)\leq \delta^2} \frac{\abs{\G_n(C)}}{\phi_q(\sqrt{P(C)})}-K\beta_{n,q}(r,\delta)\bigg)_+\\
	&\leq \max_{1\leq j\leq \ell}\bigg( \frac{ \sup_{C \in \mathscr{C}_j}\abs{\G_n(C)} }{\phi(rq^j)}- \frac{K\E \sup_{C\in \mathscr{C}_j}\abs{\G_n(C)} }{\phi(rq^j)}\bigg)_+\\
	&\leq K\max_{1\leq j\leq \ell} \frac{rq^j \sqrt{ s_j}+s_j/\sqrt{n}}{\phi(rq^j)}.
	\end{align*}
	Now the conclusion follows from $\sup_{r\leq x\leq 1}\phi(qx)/\phi(x)<\infty$.
\end{proof}

The next lemma, due to \cite[Lemma 7.2]{alexander1987rates}, provides a convenient device to derive almost sure results for ratio-type empirical processes. For any measurable function $f$, let $\sigma_P f \equiv \sqrt{Pf^2}$.

\begin{lemma}\label{lem:almost_sure_logn}
	Let $c_n,u_n$ be such that $c_n/n \downarrow$ and $u_n \downarrow$, and assume that $r_n \downarrow$ and $\sqrt{n}\delta_n \uparrow$. For a centered function class $\mathcal{F}\subset L_{2}(P)$, let
	\begin{align*}
	A_n\equiv \left\{\abs{\G_n f}\geq c_n \phi(\sigma_P f) +u_n\textrm{ for some }
	f \in \mathcal{F}, r_n\leq \sigma_P f\leq \delta_n \right\},
	\end{align*}
	and
	\begin{align*}
	A_n^{\epsilon}&\equiv \big\{\abs{\G_n f}\geq (1-\epsilon)\big(c_n \phi(\sigma_P f) +u_n\big)\textrm{ for some }
	f \in \mathcal{F}, \\
	&\qquad\qquad\qquad\qquad\qquad\qquad\qquad\qquad r_n\leq \sigma_P f\leq \sqrt{1+\epsilon}\cdot\delta_n \big\}.
	\end{align*}
	Assume that $
	\inf_{n\geq 1, t \in [r_n,\delta_n]} c_n\frac{\phi(t)}{t}>0$. 
	Then if $
	\Prob(A_n^\epsilon)=\mathcal{O}(1/(\log n)^{1+\theta})$
	holds for some $\epsilon,\theta>0$, we have $
	\Prob(A_n \textrm{ i.o.})=0$.
\end{lemma}

\begin{proof}[Proof of Theorem \ref{thm:ratio_ep_std} ]
	Consider the first claim. Note that for $0<\alpha<1$, $
	\beta_{n,q}\lesssim \max_{1\leq j\leq \ell}\frac{ (r_n q^j)^{1-\alpha}}{r_n q^j} \asymp r_n^{-\alpha}$, 
	while for $\alpha>1$, $
	\beta_{n,q}\lesssim  \max_{1\leq j\leq \ell} \frac{ n^{(\alpha-1)/2(\alpha+1)}}{r_n q^j}\asymp n^{\frac{\alpha-1}{2(\alpha+1)}}/r_n$. For $s_j\equiv s+2K\log j$, we have
	\begin{align*}
	\tau_{n,q}&\lesssim \max_{1\leq j\leq \ell} \left(\sqrt{s+2K\log j}+ \frac{s+2K\log j}{\sqrt{n}r_n q^j}\right)\\
	&\lesssim   \sqrt{s\vee \log \log(1/r_n)}+(s\vee 1)n^{-\frac{\alpha}{2(\alpha+1)}},
	\end{align*}
	and the probability estimate $
	K\sum_{j=1}^\ell \exp(-s_j/K)=Ke^{-s} \sum_{j=1}^\ell j^{-2}\leq K'e^{-s}$. 
	This proves that
	\begin{align*}
	&\Prob\bigg(\sup_{C \in \mathscr{C}: r_n^2\leq P(C)\leq 1} \frac{\abs{\G_n(C)}}{\sqrt{P(C)}}\\
	&\qquad\qquad \geq K\left(\beta_{n,q}+\sqrt{s\vee \log \log(1/r_n)}+(s\vee 1)n^{-\frac{\alpha}{2(\alpha+1)}}\right)\bigg)\leq K'e^{-s}.
	\end{align*}
	The first claim of (1) follows from Lemma \ref{lem:almost_sure_logn} by setting $s\asymp \log\log n$, and requiring $\beta_{n,q}\gg \sqrt{\log \log n} \vee \log \log n\cdot n^{-\alpha/(2(\alpha+1))}$.  The second claim follows from similar lines by observing that under (E2), Theorem \ref{thm:upper_bound_ep_set} yields that for $0<\alpha<1$, $
	\beta_{n,q}\gtrsim \max_{1\leq j\leq \ell}\frac{ (r_n q^j)^{1-\alpha}}{r_n q^j} \asymp r_n^{-\alpha}$, while for $\alpha>1$, $
	\beta_{n,q}\gtrsim  \max_{1\leq j\leq \ell} \frac{ n^{(\alpha-1)/2(\alpha+1)}}{r_n q^j}\asymp  n^{\frac{\alpha-1}{2(\alpha+1)}}/r_n$, 
	and $\tau_{n,q}$ can be estimated from above using the same arguments. 
\end{proof}

\begin{proof}[Proof of Theorem \ref{thm:ratio_ep_var}]
	
	The proof of Theorem \ref{thm:ratio_ep_var} uses a similar strategy as that of Theorem \ref{thm:ratio_ep_std}. For convenience of the reader we provide some details. Consider the first claim. Note that for $0<\alpha<1$, $
	\beta_{n,q}\lesssim \max_{1\leq j\leq \ell} \frac{(r_n q^j)^{1-\alpha}}{r_n^2 q^{2j}}\asymp r_n^{-(1+\alpha)}$, while for $\alpha>1$, $
	\beta_{n,q}\lesssim  \max_{1\leq j\leq \ell} \frac{ n^{(\alpha-1)/2(\alpha+1)}}{r_n^2 q^{2j}}\asymp r_n^{-2} n^{\frac{\alpha-1}{2(\alpha+1)}}$. 
	For $s_j\equiv s+2K\log j$, we have
	\begin{align*}
	\tau_{n,q}&\lesssim \max_{1\leq j\leq \ell} \left(r_n^{-1}\sqrt{s+2K\log j}+ \frac{s+2K\log j}{\sqrt{n}r_n^2 q^{2j}}\right)\\
	&\lesssim   \sqrt{r_n^{-2}\big(s\vee \log \log(1/r_n)\big) }+(s\vee 1)(\sqrt{n}r_n^2)^{-1}.
	\end{align*}
	This shows that, for 
	\begin{align*}
	\bar{\gamma}_n\equiv \big(r_n^{-2} n^{-\frac{1}{\alpha+1}}\big)^{1 \wedge \frac{1+\alpha}{2} }= 
	\begin{cases}
		n^{-1/2} r_n^{-(1+\alpha)}, & \alpha \in (0,1);\\
		r_n^{-2} n^{-\frac{1}{\alpha+1}}, & \alpha>1.
	\end{cases}
	\end{align*}
	we have
	\begin{align*}
	&\Prob\bigg(\sup_{C \in \mathscr{C}: r_n^2\leq P(C)\leq 1} \frac{\abs{\Prob_n(C)-P(C)}}{P(C)}\\
	&\qquad \geq K\left( \bar{\gamma}_n+  \sqrt{(nr_n^2)^{-1}\big(s\vee \log \log(1/r_n) \big)}+(s\vee 1)(nr_n^2)^{-1}\right)\bigg)\leq K'e^{-s}.
	\end{align*}
	The first claim of the theorem follows by taking $s\asymp \log  \log n$, applying Lemma \ref{lem:almost_sure_logn} and noting that $\limsup_n \bar{\gamma}_n \leq \underline{\rho}^{-1}$ by the assumption. The second claim follows similarly by estimating $\beta_{n,q}$ from below, up to a multiplicative constant, by $\bar{\gamma}_n$ and then repeat the arguments as above.
\end{proof}

\section{Proof of technical lemmas}\label{section:proof_lemma}

\subsection{Proof of Lemma \ref{lem:lse_additive_error}}

\begin{proof}[Proof of Lemma \ref{lem:lse_additive_error}]
	The proof is a modification of that of \cite[Proposition 2]{han2017sharp}. Let $
	\mathbb{M}_n \bm{1}_{C} \equiv \frac{2}{n} \sum_{i=1}^n (\bm{1}_{C}-\bm{1}_{C_0})(X_i)\xi_i - \frac{1}{n}\sum_{i=1}^n (\bm{1}_{C}-\bm{1}_{C_0})^2(X_i)$, and $
	M\bm{1}_{C} \equiv \E \left[\mathbb{M}_n(\bm{1}_{C})\right]=-P(\bm{1}_{C}-\bm{1}_{C_0})^2 = - P\abs{C\Delta C_0} $.
	Here we used the fact that $\E[\xi_i|X_i]=0$. Then it is easy to see that
	\begin{align*}
	&\abs{\mathbb{M}_n \bm{1}_{C} -\mathbb{M}_n \bm{1}_{C_0}  -(M\bm{1}_{C} - M\bm{1}_{C_0})}\\
	&\leq \biggabs{\frac{2}{n}\sum_{i=1}^n (\bm{1}_{C}-\bm{1}_{C_0})(X_i)\xi_i}+ \abs{(\Prob_n-P)(\bm{1}_{C}-\bm{1}_{C_0})^2}.
	\end{align*}
	Fix $t\geq 1$. For $j\in \N$, let $\mathscr{C}_{j}\equiv \{C \in \mathscr{C}: 2^{j-1}t\delta_n\leq  P^{1/2}(C\Delta C_0)< 2^j t\delta_n\}$. Then by a standard peeling argument, we have
	\begin{align*}
	\Prob\left( P^{1/2}\abs{\widehat{C}_n\Delta C_0}\geq t \delta_n\right)&\leq \sum_{j\geq 1} \Prob\bigg[\sup_{C \in \mathscr{C}_j}\left\{\mathbb{M}_n(\bm{1}_{C} )-\mathbb{M}_n(\bm{1}_{C_0})\right\}\geq 0 \bigg].
	\end{align*}
	Each probability term in the above display can be further bounded by
	\begin{align*}
	& \Prob\bigg[\sup_{ {C} \in \mathscr{C}_j}\left\{\mathbb{M}_n(\bm{1}_{C})-\mathbb{M}_n(\bm{1}_{C_0})-(M\bm{1}_{C}-M\bm{1}_{C_0})\right\}\geq 2^{2j-2}t^2\delta_n^2 \bigg]\\
	&\leq \Prob\bigg(\sup_{C \in \mathscr{C}: P^{1/2}\abs{C\Delta C_0} \leq 2^j t\delta_n }\biggabs{\frac{1}{\sqrt{n}}\sum_{i=1}^n \xi_i \big(\bm{1}_{C}-\bm{1}_{C_0}\big)(X_i)}\geq 2^{2j-4} t^2 \sqrt{n}\delta_n^2\bigg)\\
	&\quad\quad + \Prob\bigg(\sup_{C \in \mathscr{C}: P^{1/2}\abs{C\Delta C_0} \leq 2^j t\delta_n } \abs{\G_n \big(\bm{1}_{C}-\bm{1}_{C_0}\big)^2 }\geq 2^{2j-3} t^2 \sqrt{n}\delta_n^2\bigg).
	\end{align*}
	By the contraction principle and moment inequality for the empirical process (cf. \cite[Proposition 3.1]{gine2000exponential}), we have
	\begin{align*}
	&\E \bigg(\sup_{C \in \mathscr{C}: P^{1/2}\abs{C\Delta C_0} \leq 2^j t\delta_n }\biggabs{\frac{1}{\sqrt{n}}\sum_{i=1}^n \xi_i \big(\bm{1}_{C}-\bm{1}_{C_0}\big)(X_i)}^4\bigg)\\
	&\qquad \bigvee \E \bigg(\sup_{C \in \mathscr{C}: P^{1/2}\abs{C\Delta C_0} \leq 2^j t\delta_n }\biggabs{\frac{1}{\sqrt{n}}\sum_{i=1}^n \epsilon_i \big(\bm{1}_{C}-\bm{1}_{C_0}\big)^2(X_i)}^4\bigg)\\
	& \lesssim \big[\phi_n(2^j t \delta_n)\big]^4 + (1\vee \pnorm{\xi_1}{2})^4 2^{4j} t^4\delta_n^4+n^{-2}\big\{1\vee \E \max_{1\leq i\leq n} \abs{\xi_i}^4\big\}.
	\end{align*}
	By Chebyshev's inequality,
	\begin{align*}
	&\Prob\left(P^{1/2}\abs{\widehat{C}_n\Delta C_0}\geq t \delta_n\right)\\
	&\lesssim \sum_{j\geq 1} \bigg[ \left(\frac{\phi_n(2^j t\delta_n)}{2^{2j}t^2 \sqrt{n}\delta_n^2}\right)^4\bigvee \frac{(1\vee \pnorm{\xi_1}{2})^4 }{2^{4j}t^4 n^2\delta_n^4}\bigvee \frac{1\vee \E \max_{1\leq i\leq n} \abs{\xi_i}^4 }{2^{8j} t^8 n^{4}\delta_n^8 }\bigg].
	\end{align*}
	Under the assumption that $n\delta_n^2\geq  1 \vee \pnorm{\xi_1}{2}^2 \vee  \E^{1/4} \max_{1\leq i\leq n} \abs{\xi_i}^4 $, and noting that $\phi_n(2^j t \delta_n)\leq 2^j t \phi_n(\delta_n)$ by the assumption that $\delta \mapsto \phi_n(\delta)/\delta$ is non-increasing, the right side of the above display can be further bounded up to a constant by $
	\sum_{j\geq 1}  \left(\frac{\phi_n(\delta_n)}{2^{j}t \sqrt{n}\delta_n^2}\right)^4+\frac{1}{t^4}\lesssim \frac{1}{t^4}$ 
	for $t\geq 1$. The expectation bound follows by integrating the tail estimate.
\end{proof}

\subsection{Proof of Lemma \ref{lem:lse_dependent_bounded_error}}

\begin{proof}[Proof of Lemma \ref{lem:lse_dependent_bounded_error}]
	Clearly $\E[\xi_1|X_1]=0$ and $\E[\xi_1^2|X_1]\lesssim 1\vee \pnorm{\xi_1}{2}^2$ almost surely, so we only need to verify that the assumed conditions imply (\ref{ineq:ep_additive_error}) (in particular the second one). By symmetrization (cf. \cite[Theorem 3.1.21]{gine2015mathematical}) and contraction principle (cf. \cite[Theorem 3.1.17]{gine2015mathematical}) for the empirical process,
	\begin{align*}
	&\E \sup_{C \in \mathscr{C}: P\abs{C\Delta C_0}<\delta^2} \biggabs{\sum_{i=1}^n (\bm{1}_C-\bm{1}_{C_0})(X_i)\xi_i}\\ 
	&\lesssim \E \sup_{C \in \mathscr{C}: P\abs{C\Delta C_0}<\delta^2}\biggabs{\sum_{i=1}^n \epsilon_i (\bm{1}_C-\bm{1}_{C_0})(X_i)f_{C_0}(X_i) (\eta_i-2a) }\\
	&\lesssim \E \sup_{C \in \mathscr{C}: P\abs{C\Delta C_0}<\delta^2}\biggabs{\sum_{i=1}^n \epsilon_i (\bm{1}_C-\bm{1}_{C_0})(X_i)f_{C_0}(X_i) }\\
	&\lesssim \E \sup_{C \in \mathscr{C}: P\abs{C\Delta C_0}<\delta^2}\biggabs{\sum_{i=1}^n \epsilon_i (\bm{1}_C-\bm{1}_{C_0})(X_i) }\\
	&\qquad\qquad \bigvee \E \sup_{C \in \mathscr{C}: P\abs{C\Delta C_0}<\delta^2}\biggabs{\sum_{i=1}^n \epsilon_i (\bm{1}_{C\cap C_0}-\bm{1}_{C_0})(X_i) },
	\end{align*}
	where the last inequality follows by noting that
	\begin{align*}
	(\bm{1}_C-\bm{1}_{C_0})f_{C_0} = (\bm{1}_C-\bm{1}_{C_0})(2\bm{1}_{C_0}-1) = 2(\bm{1}_{C\cap C_0}-\bm{1}_{C_0})-(\bm{1}_C-\bm{1}_{C_0}),
	\end{align*}
	and using triangle inequality. The proof is complete.
\end{proof}

\section{Some useful tools}\label{section:tools}

\subsection{Talagrand's concentration inequality}
We frequently use Talagrand's concentration inequality \cite{talagrand1996new} in the following form given by \cite[Theorems 3.3.9/3.3.10]{gine2015mathematical}. 

\begin{lemma}
Let $\mathcal{F}$ be a countable class of real-valued measurable functions such that $\sup_{f \in \mathcal{F}} \pnorm{f}{\infty}\leq b$. Then
\begin{align*}
\Prob\bigg(\biggabs{\sup_{f \in \mathcal{F}}\abs{\G_n (f)} -\E\sup_{f \in \mathcal{F}}\abs{\G_n (f)}}\geq  \sqrt{2\sigma_n^2 x}+4 b \frac{x}{\sqrt{n}} \bigg)\leq 2e^{-x},
\end{align*}
where $\sigma_n^2\equiv 2 n^{-1/2}\E\sup_{f \in \mathcal{F}}\abs{\G_n (f)}+\sup_{f \in \mathcal{F}} \mathrm{Var}_P (f)$.
\end{lemma}

\subsection{Sudakov minorization}

The following Sudakov minorization will be useful.
\begin{lemma}[Sudakov minorization \cite{sudakov1969gauss}]\label{lem:sudakov_minor}
	Let $(X_t)_{t \in T}$ be a centered separable Gaussian process, and $\pnorm{t-s}{}^2:=\E(X_t-X_s)^2$. Then
	\begin{align*}
	\E \sup_{t \in T} X_t \geq C^{-1} \sup_{\epsilon>0} \epsilon \sqrt{\log \mathcal{N}(\epsilon,T,\pnorm{\cdot}{})}.
	\end{align*}
	Here $C>0$ is a universal constant.
\end{lemma}

\subsection{Multiplier inequalities}

We use the following multiplier inequality due to \cite[Theorem 1]{han2017sharp}.

\begin{lemma}\label{lem:multiplier_ineq}
	Suppose that $\xi_1,\ldots,\xi_n$ are i.i.d. mean-zero random variables independent of i.i.d. $X_1,\ldots,X_n$. Let $\{\mathcal{F}_k\}_{k=1}^n$ be a sequence of function classes such that $\mathcal{F}_k \supset \mathcal{F}_n$ for all $1\leq k\leq n$.  Assume further that there exist non-decreasing concave functions $\{\psi_n\}:\R_{\geq 0}\to \R_{\geq 0}$ with $\psi_n(0)=0$ such that
	\begin{align*}
	\E \sup_{f \in \mathcal{F}_k}\biggabs{\sum_{i=1}^k\epsilon_i f(X_i)}\leq \psi_n(k)
	\end{align*}
	holds for all $1\leq k\leq n$. Then
	\begin{align*}
	\E \sup_{f \in \mathcal{F}_n}\biggabs{\sum_{i=1}^n \xi_i f(X_i)}\leq 4\int_0^\infty \psi_n\big(n\cdot \Prob(\abs{\xi_1}>t)\big)\ \d{t}.
	\end{align*}
\end{lemma}

The following alternative formulation of the multiplier inequality, proved in \cite[Proposition 1]{han2017sharp}, will also be useful.
\begin{lemma}\label{lem:interpolation_ineq}
	Let $\xi_1,\ldots,\xi_n$ be i.i.d. symmetric mean-zero multipliers independent of i.i.d. samples $X_1,\ldots,X_n$. For any function class $\mathcal{F}$,
	\begin{align*}
	\E \sup_{f \in \mathcal{F}}\biggabs{\sum_{i=1}^n \xi_i f(X_i)}\leq  \E \left[ \sum_{k=1}^n  (\abs{\xi_{(k)}}-\abs{\xi_{(k+1)}}) \E\sup_{f \in \mathcal{F}}\biggabs{\sum_{i=1}^k \epsilon_i f(X_i)}\right]
	\end{align*}
	where $\abs{\xi_{(1)}}\geq \cdots \geq \abs{\xi_{(n)}}\geq \abs{\xi_{(n+1)}}\equiv 0$ are the reversed order statistics for $\{\xi_i\}_{i=1}^n$, and $\epsilon_i$'s are i.i.d. Rademacher random variables.
\end{lemma}

\section*{Acknowledgments}
The major part of this work (materials presented before Section \ref{section:binary_classification} and their proofs) is based on Chapter 4 of the author's University of Washington Ph.D. thesis in 2018. The author would like to thank Jon Wellner and Cun-Hui Zhang for helpful discussion and encouragements. He would also like to thank four referees, an Associate Editor and the Editor for their very helpful comments and suggestions that significantly improved the quality of the paper. 

\bibliographystyle{amsalpha}
\bibliography{mybib}

\providecommand{\bysame}{\leavevmode\hbox to3em{\hrulefill}\thinspace}
\providecommand{\MR}{\relax\ifhmode\unskip\space\fi MR }
\providecommand{\MRhref}[2]{%
  \href{http://www.ams.org/mathscinet-getitem?mr=#1}{#2}
}
\providecommand{\href}[2]{#2}
\begin{thebibliography}{HWCS19}

\bibitem[Ale87]{alexander1987rates}
Kenneth~S. Alexander, \emph{Rates of growth and sample moduli for weighted
  empirical processes indexed by sets}, Probab. Theory Related Fields
  \textbf{75} (1987), no.~3, 379--423. \MR{890285}

\bibitem[Bar16]{baraud2016bounding}
Yannick Baraud, \emph{Bounding the expectation of the supremum of an empirical
  process over a (weak) {VC}-major class}, Electron. J. Stat. \textbf{10}
  (2016), no.~2, 1709--1728. \MR{3522658}

\bibitem[BBM99]{barron1999risk}
Andrew Barron, Lucien Birg{\'e}, and Pascal Massart, \emph{Risk bounds for
  model selection via penalization}, Probab. Theory Related Fields \textbf{113}
  (1999), no.~3, 301--413. \MR{1679028 (2000k:62049)}

\bibitem[Bir83]{birge1983approximation}
Lucien Birg\'e, \emph{Approximation dans les espaces m\'etriques et th\'eorie
  de l'estimation}, Z. Wahrsch. Verw. Gebiete \textbf{65} (1983), no.~2,
  181--237. \MR{722129}

\bibitem[BM93]{birge1993rates}
Lucien Birg{\'e} and Pascal Massart, \emph{Rates of convergence for minimum
  contrast estimators}, Probab. Theory Related Fields \textbf{97} (1993),
  no.~1-2, 113--150. \MR{1240719 (94m:62095)}

\bibitem[Bru13]{brunel2013adaptive}
Victor-Emmanuel Brunel, \emph{Adaptive estimation of convex polytopes and
  convex sets from noisy data}, Electron. J. Stat. \textbf{7} (2013),
  1301--1327. \MR{3063609}

\bibitem[CDSS18]{carpenter2018near}
Timothy Carpenter, Ilias Diakonikolas, Anastasios Sidiropoulos, and Alistair
  Stewart, \emph{Near-optimal sample complexity bounds for maximum likelihood
  estimation of multivariate log-concave densities}, Conference on Learning
  Theory, 2018, pp.~1--29.

\bibitem[CGS18]{chatterjee2018matrix}
Sabyasachi Chatterjee, Adityanand Guntuboyina, and Bodhisattva Sen, \emph{On
  matrix estimation under monotonicity constraints}, Bernoulli \textbf{24}
  (2018), no.~2, 1072--1100. \MR{3706788}

\bibitem[Cha14]{chatterjee2014new}
Sourav Chatterjee, \emph{A new perspective on least squares under convex
  constraint}, Ann. Statist. \textbf{42} (2014), no.~6, 2340--2381.
  \MR{3269982}

\bibitem[DGL96]{devroye1996probabilistic}
Luc Devroye, L\'aszl\'o Gy\"orfi, and G\'abor Lugosi, \emph{A {P}robabilistic
  {T}heory of {P}attern {R}ecognition}, Applications of Mathematics (New York),
  vol.~31, Springer-Verlag, New York, 1996. \MR{1383093}

\bibitem[DK19]{dagan2019log}
Yuval Dagan and Gil Kur, \emph{The log-concave maximum likelihood estimator is
  optimal in high dimensions}, arXiv preprint arXiv:1903.05315 (2019).

\bibitem[Dud82]{dudley1982empirical}
R.~M. Dudley, \emph{Empirical and {P}oisson processes on classes of sets or
  functions too large for central limit theorems}, Z. Wahrsch. Verw. Gebiete
  \textbf{61} (1982), no.~3, 355--368. \MR{679680}

\bibitem[Dud14]{dudley1999uniform}
\bysame, \emph{Uniform central limit theorems}, second ed., Cambridge Studies
  in Advanced Mathematics, vol. 142, Cambridge University Press, New York,
  2014. \MR{3445285}

\bibitem[DW16]{doss2013global}
Charles~R. Doss and Jon~A. Wellner, \emph{Global rates of convergence of the
  {MLE}s of log-concave and {$s$}-concave densities}, Ann. Statist. \textbf{44}
  (2016), no.~3, 954--981. \MR{3485950}

\bibitem[GK06]{gine2006concentration}
Evarist Gin{\'e} and Vladimir Koltchinskii, \emph{Concentration inequalities
  and asymptotic results for ratio type empirical processes}, Ann. Probab.
  \textbf{34} (2006), no.~3, 1143--1216. \MR{2243881}

\bibitem[GKW03]{gine2003ratio}
Evarist Gin{\'e}, Vladimir Koltchinskii, and Jon~A. Wellner, \emph{Ratio limit
  theorems for empirical processes}, Stochastic inequalities and applications,
  Progr. Probab., vol.~56, Birkh\"auser, Basel, 2003, pp.~249--278.
  \MR{2073436}

\bibitem[GLZ00]{gine2000exponential}
Evarist Gin{\'e}, Rafa{\l} Lata{\l}a, and Joel Zinn, \emph{Exponential and
  moment inequalities for {$U$}-statistics}, High dimensional probability, {II}
  ({S}eattle, {WA}, 1999), Progr. Probab., vol.~47, Birkh\"auser Boston,
  Boston, MA, 2000, pp.~13--38. \MR{1857312}

\bibitem[GN16]{gine2015mathematical}
Evarist Gin\'{e} and Richard Nickl, \emph{Mathematical foundations of
  infinite-dimensional statistical models}, Cambridge Series in Statistical and
  Probabilistic Mathematics, [40], Cambridge University Press, New York, 2016.
  \MR{3588285}

\bibitem[Gre81]{grenander1981abstract}
Ulf Grenander, \emph{Abstract {I}nference}, John Wiley \&\ Sons, Inc., New
  York, 1981, Wiley Series in Probability and Mathematical Statistics.
  \MR{599175}

\bibitem[Gun12]{guntuboyina2012optimal}
Adityanand Guntuboyina, \emph{Optimal rates of convergence for convex set
  estimation from support functions}, Ann. Statist. \textbf{40} (2012), no.~1,
  385--411. \MR{3014311}

\bibitem[GW07]{gao2007entropy}
Fuchang Gao and Jon~A. Wellner, \emph{Entropy estimate for high-dimensional
  monotonic functions}, J. Multivariate Anal. \textbf{98} (2007), no.~9,
  1751--1764. \MR{2392431}

\bibitem[HW16]{han2015approximation}
Qiyang Han and Jon~A. Wellner, \emph{Approximation and estimation of
  {$s$}-concave densities via {R}\'{e}nyi divergences}, Ann. Statist.
  \textbf{44} (2016), no.~3, 1332--1359. \MR{3485962}

\bibitem[HW19]{han2017sharp}
\bysame, \emph{Convergence rates of least squares regression estimators with
  heavy-tailed errors}, Ann. Statist. \textbf{47} (2019), no.~4, 2286--2319.
  \MR{3953452}

\bibitem[HWCS19]{han2017isotonic}
Qiyang Han, Tengyao Wang, Sabyasachi Chatterjee, and Richard~J. Samworth,
  \emph{Isotonic regression in general dimensions}, Ann. Statist. \textbf{47}
  (2019), no.~5, 2440--2471. \MR{3988762}

\bibitem[KGGS20]{kur2020convex}
Gil Kur, Fuchang Gao, Adityanand Guntuboyina, and Bodhisattva Sen, \emph{Convex
  regression in multidimensions: Suboptimality of least squares estimators},
  arXiv preprint arXiv:2006.02044 (2020).

\bibitem[Kol06]{koltchinskii2006local}
Vladimir Koltchinskii, \emph{Local {R}ademacher complexities and oracle
  inequalities in risk minimization}, Ann. Statist. \textbf{34} (2006), no.~6,
  2593--2656. \MR{2329442}

\bibitem[KS16]{kim2016global}
Arlene K.~H. Kim and Richard~J. Samworth, \emph{Global rates of convergence in
  log-concave density estimation}, Ann. Statist. \textbf{44} (2016), no.~6,
  2756--2779. \MR{3576560}

\bibitem[KT92]{korostelev1992asymtotically}
A.~P. Korostel\"ev and A.~B. Tsybakov, \emph{Asymptotically minimax image
  reconstruction problems}, Topics in nonparametric estimation, Adv. Soviet
  Math., vol.~12, Amer. Math. Soc., Providence, RI, 1992, pp.~45--86.
  \MR{1191691}

\bibitem[KT93]{korestelev1993minimax}
\bysame, \emph{Minimax theory of image reconstruction}, Lecture Notes in
  Statistics, vol.~82, Springer-Verlag, New York, 1993. \MR{1226450}

\bibitem[LC73]{lecam1973convergence}
Lucien Le~Cam, \emph{Convergence of estimates under dimensionality
  restrictions}, Ann. Statist. \textbf{1} (1973), 38--53. \MR{0334381}

\bibitem[Lec07]{lecue2007simultaneous}
Guillaume Lecu\'e, \emph{Simultaneous adaptation to the margin and to
  complexity in classification}, Ann. Statist. \textbf{35} (2007), no.~4,
  1698--1721. \MR{2351102}

\bibitem[MN06]{massart2006risk}
Pascal Massart and \'Elodie N\'ed\'elec, \emph{Risk bounds for statistical
  learning}, Ann. Statist. \textbf{34} (2006), no.~5, 2326--2366. \MR{2291502}

\bibitem[MSW83]{mason1983strong}
David~M. Mason, Galen~R. Shorack, and Jon~A. Wellner, \emph{Strong limit
  theorems for oscillation moduli of the uniform empirical process}, Z.
  Wahrsch. Verw. Gebiete \textbf{65} (1983), no.~1, 83--97. \MR{717935}

\bibitem[MT95]{mammen1995asymptotical}
Enno Mammen and Alexandre~B. Tsybakov, \emph{Asymptotical minimax recovery of
  sets with smooth boundaries}, Ann. Statist. \textbf{23} (1995), no.~2,
  502--524. \MR{1332579}

\bibitem[MT99]{mammen1999smooth}
\bysame, \emph{Smooth discrimination analysis}, Ann. Statist. \textbf{27}
  (1999), no.~6, 1808--1829. \MR{1765618}

\bibitem[Oss87]{ossiander1987central}
Mina Ossiander, \emph{A central limit theorem under metric entropy with {$L_2$}
  bracketing}, Ann. Probab. \textbf{15} (1987), no.~3, 897--919. \MR{893905}

\bibitem[Pol02]{pollard2002maximal}
David Pollard, \emph{Maximal inequalities via bracketing with adaptive
  truncation}, Ann. Inst. H. Poincar\'{e} Probab. Statist. \textbf{38} (2002),
  no.~6, 1039--1052. \MR{1955351}

\bibitem[RWD88]{robertson1988order}
Tim Robertson, F.~T. Wright, and R.~L. Dykstra, \emph{Order restricted
  statistical inference}, Wiley Series in Probability and Mathematical
  Statistics: Probability and Mathematical Statistics, John Wiley \& Sons,
  Ltd., Chichester, 1988. \MR{961262}

\bibitem[Sau12]{saumard2012optimal}
Adrien Saumard, \emph{Optimal upper and lower bounds for the true and empirical
  excess risks in heteroscedastic least-squares regression}, Electron. J. Stat.
  \textbf{6} (2012), 579--655. \MR{2988421}

\bibitem[Stu82]{stute1982oscillation}
Winfried Stute, \emph{The oscillation behavior of empirical processes}, Ann.
  Probab. \textbf{10} (1982), no.~1, 86--107. \MR{637378}

\bibitem[Stu84]{stute1984oscillation}
\bysame, \emph{The oscillation behavior of empirical processes: the
  multivariate case}, Ann. Probab. \textbf{12} (1984), no.~2, 361--379.
  \MR{735843}

\bibitem[Sud69]{sudakov1969gauss}
V.~N. Sudakov, \emph{Gauss and {C}auchy measures and {$\varepsilon$}-entropy},
  Dokl. Akad. Nauk SSSR \textbf{185} (1969), 51--53. \MR{0247034}

\bibitem[SW82]{shorack1982limit}
Galen~R. Shorack and Jon~A. Wellner, \emph{Limit theorems and inequalities for
  the uniform empirical process indexed by intervals}, Ann. Probab. \textbf{10}
  (1982), no.~3, 639--652. \MR{659534}

\bibitem[SW10]{seregin2010nonparametric}
Arseni Seregin and Jon~A. Wellner, \emph{Nonparametric estimation of
  multivariate convex-transformed densities}, Ann. Statist. \textbf{38} (2010),
  no.~6, 3751--3781. \MR{2766867 (2012b:62126)}

\bibitem[Tal96]{talagrand1996new}
Michel Talagrand, \emph{New concentration inequalities in product spaces},
  Invent. Math. \textbf{126} (1996), no.~3, 505--563. \MR{1419006}

\bibitem[Tsy04]{tsybakov2004optimal}
Alexandre~B. Tsybakov, \emph{Optimal aggregation of classifiers in statistical
  learning}, Ann. Statist. \textbf{32} (2004), no.~1, 135--166. \MR{2051002}

\bibitem[vdG87]{van1987new}
Sara van~de Geer, \emph{A new approach to least-squares estimation, with
  applications}, Ann. Statist. \textbf{15} (1987), no.~2, 587--602. \MR{888427}

\bibitem[vdG90]{van1990estimating}
\bysame, \emph{Estimating a regression function}, Ann. Statist. \textbf{18}
  (1990), no.~2, 907--924. \MR{1056343}

\bibitem[vdG93]{van1993hellinger}
\bysame, \emph{Hellinger-consistency of certain nonparametric maximum
  likelihood estimators}, Ann. Statist. \textbf{21} (1993), no.~1, 14--44.
  \MR{1212164}

\bibitem[vdG95]{van1995method}
\bysame, \emph{The method of sieves and minimum contrast estimators}, Math.
  Methods Statist. \textbf{4} (1995), no.~1, 20--38. \MR{1324688}

\bibitem[vdG00]{van2000empirical}
\bysame, \emph{Applications of {E}mpirical {P}rocess {T}heory}, Cambridge
  Series in Statistical and Probabilistic Mathematics, vol.~6, Cambridge
  University Press, Cambridge, 2000. \MR{1739079 (2001h:62002)}

\bibitem[vdGW17]{van2015concentration}
Sara van~de Geer and Martin~J. Wainwright, \emph{On concentration for
  (regularized) empirical risk minimization}, Sankhya A \textbf{79} (2017),
  no.~2, 159--200. \MR{3707417}

\bibitem[vdV96a]{van1996efficient}
Aad van~der Vaart, \emph{Efficient maximum likelihood estimation in
  semiparametric mixture models}, Ann. Statist. \textbf{24} (1996), no.~2,
  862--878. \MR{1394993}

\bibitem[vdV96b]{van1996new}
\bysame, \emph{New {D}onsker classes}, Ann. Probab. \textbf{24} (1996), no.~4,
  2128--2140. \MR{1415244}

\bibitem[vdVW96]{van1996weak}
Aad van~der Vaart and Jon~A. Wellner, \emph{Weak {C}onvergence and {E}mpirical
  {P}rocesses}, Springer Series in Statistics, Springer-Verlag, New York, 1996.
  \MR{1385671 (97g:60035)}

\bibitem[Wel78]{wellner1978limit}
Jon~A. Wellner, \emph{Limit theorems for the ratio of the empirical
  distribution function to the true distribution function}, Z. Wahrsch. Verw.
  Gebiete \textbf{45} (1978), no.~1, 73--88. \MR{0651392}

\bibitem[WS95]{wong1995probability}
Wing~Hung Wong and Xiaotong Shen, \emph{Probability inequalities for likelihood
  ratios and convergence rates of sieve {MLE}s}, Ann. Statist. \textbf{23}
  (1995), no.~2, 339--362. \MR{1332570}

\bibitem[YB99]{yang1999information}
Yuhong Yang and Andrew Barron, \emph{Information-theoretic determination of
  minimax rates of convergence}, Ann. Statist. \textbf{27} (1999), no.~5,
  1564--1599. \MR{1742500}

\end{thebibliography}

\end{document}